\documentclass{amsart}

\usepackage[top=4cm, bottom=3cm,left=2.5cm,right=2.5cm]{geometry}
\usepackage[font=small,format=plain,labelfont=bf,up,textfont=it,up]{caption}
\usepackage{amsthm,amsmath,amsfonts,amssymb,array,enumerate,verbatim}
\usepackage[all]{xy}
\usepackage{url}
\usepackage{amsmath}
\usepackage{amssymb}
\usepackage{amsthm}
\usepackage{mathtools}
\usepackage{latexsym}
\usepackage{centernot}
\usepackage{bbm}
\usepackage{braket}

\usepackage{hyperref}

\usepackage{enumerate} 
\usepackage{color}
\usepackage{mathrsfs}
\usepackage{tikz}
\usetikzlibrary{matrix,arrows,decorations.pathmorphing}
\usepackage{tikz-cd}

\DeclareMathOperator{\Hom}{Hom}

\DeclareMathOperator{\Ind}{Ind}
\DeclareMathOperator{\cInd}{c-Ind}
\DeclareMathOperator{\End}{End}

\DeclareMathOperator{\Gal}{Gal}
\DeclareMathOperator{\GL}{GL}
\DeclareMathOperator{\id}{id}
\newcommand{\abs}[1]{\left|#1\right|}    % absolute value
\DeclareMathOperator{\rk}{rk}

\newtheorem{theorem}{Theorem}[section]
\newtheorem{lemma}[theorem]{Lemma}
\newtheorem{proposition}[theorem]{Proposition}

\newtheorem{definition}[theorem]{Definition}
\newtheorem*{remark}{Remark}
\newtheorem{example}[theorem]{Example}

\begin{document}

\title{The Doubling Method in Algebraic Families}
\author{Johannes Girsch}
\address{Departement of Mathematics\\ Imperial College \\ London SW7 2AZ \\ United Kingdom}
\email{johannes.girsch@live.de}

\begin{abstract}

We define the doubling zeta integral for smooth families of representations of classical groups. Following this we prove a rationality result for these zeta integrals and show that they satisfy a functional equation. Moreover, we show that there exists an apropriate normalizing factor which allows us to construct $\gamma$-factors for smooth families out of the functional equation. We prove that under certain hypothesis, specializing this $\gamma$-factor at a point of the family yields the $\gamma$-factor defined by Piateski-Shapiro and Rallis.
\end{abstract}

\maketitle

\section{Introduction}

We fix a prime number $p$ and a finite extension $F$ of $\mathbb Q_p$ whose residue field has cardinality $q$. An important topic in the representation theory of reductive $p$-adic groups over $\mathbb C$ is the definition of local constants, which is also pivotal in the formulation of the local Langlands correspondence. In recent years there has been progress in defining such local constants in the modular setting. For example, Minguez and Kurinczuk-Matringe defined $\ell$-modular $L$-functions via the theory of Godement-Jacquet integrals in \cite{minguez2012fonctions}, respectively the Rankin-Selberg method in \cite{kurinczuk2017rankin}. Moss observed in \cite{moss2016interpolating} that $L$-functions do not seem to behave well for representations with coefficients in more general noetherian rings. However, he was able to define Rankin-Selberg zeta integrals and gamma factors for representations (over general noetherian rings) which are "Co-Whittaker" (c.f. \cite{helm2016whittaker}) and proved that they interpolate the classical constants over $\mathbb C$. In subsequent work (\cite{moss2015gamma}) he was able to extend this to associate gamma factors to representations of the pair $\GL_n\times\GL_m$ in this general setting and proved, together with Helm, a converse theorem which turned out to be crucial for the proof of local Langlands conjecture in families for $\GL_n$ in \cite{helm2018converse}.\\
So far all the constructions mentioned above have only been worked out for $\GL_n(F)$ or inner forms thereof. Recently Helm and his collaborators switched focus to investigate the local Langlands correspondence in families for reductive groups other than $\GL_n$, and a theory of gamma factors in families for such representations would be a favourable thing. In this article we develop the doubling method of Piateski-Shapiro and Rallis (\cite{piatetski1986varepsilon},\cite{piatetski1987functions}) over general noetherian coefficient rings and a wide class of classical groups. Moreover, this approach, as the doubling method over $\mathbb C$, is independent of any theory of Whittaker models, which is in contrast to the work of Moss.

We give a brief, but incomplete, description of the doubling method for representations over $\mathbb C$ following \cite{lapid2005local}. Roughly, if $G$ is the isometry group of a vector space over $F$ equipped with a nondegenerate bilinear form one constructs a group $G^\Box$ which admits an embedding $G\times G\hookrightarrow G^\Box$ and satisfies some additional favourable properties. One has a natural maximal parabolic $P$ of $G^\Box$ together with an $F^\times$-valued character $\Delta$ of $P$. For a character $\omega$ of $F^\times$ and a complex parameter $s\in\mathbb C$ set $\omega_s\coloneqq\omega\cdot\abs{.}^s_F$ and let $I(s,\omega)\coloneqq i_P^{G^\Box}(\omega_s\circ\Delta)$. For $f\in I(0,\omega)$ and $s\in\mathbb C$ we define $f_s\in I(s,\omega)$ to be $f\cdot\abs{\Delta}_F^s$. Let $(\pi,V)$ be a smooth irreducible representation of $G$ over $\mathbb C$. Then for $v\in\pi$ and $\varphi$ in the contragredient $\widetilde{\pi}$ one defines the zeta integral
$$Z(f_s,v,\varphi)\coloneqq\int_{G}f_s(g,1)\varphi(\pi(g)v)\mathrm dg.$$
The following theorem holds.
\begin{theorem}[Theorem 4.1 in \cite{yamana2014functions}]
	The above defined integral converges for $s\in\mathbb C$ with large enough real part and is an element of  $\mathbb C(q^{-s})$. Moreover, for a certain intertwining operator $M(s)\colon I(s,\omega)\to I(-s,\omega^{-1})$ the functional equation
	$$Z(M(s)f_s,v,\varphi)=\Gamma(s,\pi,\omega)Z(f_s,v,\varphi)$$
	holds for an appropriate scalar-valued function $\Gamma(s,\pi,\omega)$.
\end{theorem}\noindent
By using $\Gamma(s,\pi,\omega)$ and a normalized version of the intertwining operator $M^*(s)$ (which depends on the choice of an additive character $\psi\colon F\to\mathbb C^\times$) one can then define the gamma factor $\gamma(s,\pi\times\omega,\psi)$ associated to the pair $\pi\times\omega$.\\
In this classical setting over $\mathbb C$, the standard proof of the above theorems is roughly along the following lines. First, one shows a certain multiplicity one statement, namely that $\Hom_{G\times G}(I(\omega,s),\widetilde{\pi}\otimes(\pi\otimes\omega))$ is a one dimensional $\mathbb C$-vector space for all but finitely many complex numbers $s$. Then one proves that the doubling zeta integral converges for $\operatorname{Re}(s)$ large enough, in which case the doubling zeta integral defines an element of the aformentioned $\Hom$-space. From this, Bernstein's continuation principle (\cite{banks1998corollary}) yields the rationality and the functional equation of the doubling zeta integral.\\
We now state the main results of this article. For notational convenience we assume that $G=\operatorname{Sp}_n(F)$. Let $A$ be a noetherian $\mathbb Z[1/p]$-algebra and $S$ the multiplicative subset of $A[X,X^{-1}]$ consisting of Laurent polynomials whose leading and tailing coefficient is a unit.
\begin{theorem}[Rationality, Theorem \ref{theorat}]\label{intrat}
Let $(\pi,V)$ be a smooth, admissible, $G$-finite $A[G]$-module and $\omega\colon F^\times\to A$ a smooth character. Then for all $f\in I(X,\omega)$ (which is the analogue of $I(s,\omega)$) and $v\in V,\varphi\in\widetilde{V}$ the doubling zeta integral $$Z(X,\varphi,f)\coloneqq\int_Gf(g,1)\varphi(\pi(g)v)dg$$ yields an element of $S^{-1}A[X,X^{-1}]$. 
\end{theorem}\noindent
Since convergence does not make sense in our setting and we do not have an analogue of Bernstein's continuation principle we need a different approach. However, our proof of rationality is similar to and motivated by the proof of the aformentioned multiplicity one theorem. Namely, we show the above result by using a filtration on the space $I(X,\omega)$  which then gives rise to a filtration on $I(X,\omega)\otimes(V\otimes\widetilde{V})$. The filtration arises, similarly to the geometric Lemma, from the action of $G\times G$ on the flag variety $P\backslash G^\Box$ and then via support of elements in $I(X,\omega)$ on different orbits. We show that the successive quotients are annihilated by Laurent polynomials in $S$ which then implies the above result.

We now state our version of the functional equation and give an idea of the proof. 
\begin{theorem}[Functional equation, Theorem \ref{functionaltheorem}]\label{intfunc}
	Let $(\pi,V)$ be a smooth, admissible, $G$-finite $A[G]$-module and $\omega\colon F^\times\to A$ a smooth character. Suppose that the natural map $V\otimes\widetilde{V}\to A$ is surjective and that $\Hom_{G\times G}(V\otimes\widetilde{V},A)\cong A$. There exists an intertwining operator $M_X$ and an element $\Gamma(X,\pi,\omega)\in S^{-1}A[X,X^{-1}]$ such that the functional equation
	$$Z(X^{-1},\varphi,M_X(f))=\Gamma(X,\pi,\omega)Z(X,\varphi,f)$$
	holds for all $\varphi\in V\otimes\widetilde{V}$ and $f\in I(X,\omega)$.
\end{theorem}\noindent To prove this result we first construct an intertwining operator $M(X)\colon I(X,\omega)\to S^{-1}A[X,X^{-1}]\otimes I(X^{-1},\omega^{-1})$. If we multiply $Z(X^{-1},\_,M_X(\_))$ and $Z(X,\_,\_)$ by certain Laurent polynomials in $S$ they both become elements of $$\operatorname{Hom}_{A[X,X^{-1}]}\left((V\otimes\widetilde{V})\otimes_{A[G\times G]}I(X,\omega)^{(0)},A[X,X^{-1}]\right),$$
where $I(X,\omega)^{(0)}$ is the bottom element of the filtration on $I(X,\omega)$. We then show that, under the technical hypothesis of the above theorem, this $\operatorname{Hom}$-space is free of rank one as an $A[X,X^{-1}]$-module. It is not hard to see that one can choose input data such that the doubling zeta integral becomes a constant from which the result follows. For inner forms we need additional assumptions on the representation, which can be removed for quasi-split groups by a result of Dat. Conjecturally however the results of Dat can be generalized to arbitrary inner forms which would allow us to remove these additional hypotheses. 

Let $W(k)$ be the Witt vectors of an algebraically closed field $k$ with characteristic $\ell\not=p$ and assume that $A$ is an $W(k)$-algebra. We also prove that the normalized intertwining operator $M^*(s)$ interpolates well in families which then allows us to define the normalized gamma factor $\gamma(X,\pi\times\omega,\psi)\in S^{-1}A[X,X^{-1}]$ associated to the pair $\pi\times\omega$ and an additive character $\psi\colon F\to W(k)^\times$.

We give a quick overview of the different sections in this paper. In Section \ref{sectionnotpre} we introduce some representation theoretic preliminaries that will be of use later. In the following section we describe the classical groups we will be working with. In Section \ref{sectiondoublingzeta} we introduce the doubling method and adapt it to our setting. In Section \ref{sectionrational} we prove the rationality of the doubling zeta integral. Following this in Section \ref{sectionintertwining} we will construct the intertwining operator which we then use to prove the functional equation in Section \ref{sectionfunctional}. Here we also give an example of a class of representations which are related to the local Langlands conjecture in families and satisfy Theorems \ref{intrat} and \ref{intfunc}. In the last section we show that the normalizing factor for the intertwining operator can be interpolated as well and lastly give a definition of the normalized gamma factor.

We hope we can extend the methods used in this article to show that the recent extension of the doubling method to pairs $G\times\GL_n$ by Cai, Friedberg, Ginzburg and Kaplan in \cite{cai2019doubling} interpolates well. For applications to local Langlands in families we would then hope to prove a converse theorem as in \cite{helm2018converse}. The main references for the doubling method we used are \cite{piatetski1987functions}, \cite{harris1996theta}, \cite{lapid2005local}, \cite{yamana2014functions} and \cite{kakuhama2019local}.

\subsection{Acknowledgements}
I am very grateful to my supervisor David Helm for suggesting this problem and invaluable guidance throughout. Moreover, I would like to thank Erez Lapid and Shunsuke Yamana for answering questions about the doubling method over the complex numbers. I would also like to thank Sascha Biberhofer, Robert Kurinczuk and Gil Moss for useful discussions. This work was supported by the Engineering and Physical Sciences Research Council [EP/L015234/1], The EPSRC Centre for Doctoral Training in Geometry and Number Theory (The London School of Geometry and Number Theory), University College London and King's College London.

\section{Notation and Preliminaries}\label{sectionnotpre} Let $p$ be a prime number and fix some finite extension $F$ of $\mathbb Q_p$. Moreover, $E$ will always be a finite extension of $F$, with ring of integers $\mathcal O_E$, uniformizer $\varpi_E$, and its residue field has cardinality $q_E$.
We denote by $D$ a central division algebra over $E$. For a free left $D$-module $W$ let $\operatorname{Nrd}_W$ (respectively $\operatorname{Trd}_W$) be the reduced norm (respectively the reduced trace) of the central simple $E$-algebra $\End_D(W)$.\\ 
We fix a commutative noetherian ring $A$, which is also a $\mathbb Z[1/p]$-algebra. Since we deal with normalized inductions we assume that $A$ contains a square root of $p$ and we fix a choice of such. For any $\mathfrak p\in\operatorname{Spec}(A)$ we set $\kappa(\mathfrak p)\coloneqq A_{\mathfrak p}/\mathfrak pA_{\mathfrak p}$. Let $S$ be the multiplicative subset of the ring of Laurent polynomials $A[X,X^{-1}]$ which consists of Laurent polynomials whose leading and trailing coefficients are units in $A$. Later we will also sometimes assume that $A$ is a $W(k)$-algebra where $W(k)$ is the ring of Witt vectors of some algebraically closed field of characteristic $\ell\not=p$.\\
Suppose that $G$ is a locally profinite group. Then a smooth representation of $G$ over $A$ is an $A[G]$-module where every element is stabilized by an open compact subgroup of $G$. We say an $A[G]$-module is $G$-finite if it is finitely generated as an $A[G]$-module. Moreover, an $A[G]$-module will be called admissible if for any open compact subgroup $K$ of $G$ the submodule of vectors which are fixed under the action of $K$ is finitely generated as an $A$-module. For a smooth representation $(\pi,V)$ we denote by $(\widetilde{\pi},\widetilde{V})$ the contragredient of $\pi$ which consists of the smooth vectors in $\operatorname{Hom}_A(V,A)$.\\
Let $H$ be a closed subgroup of $G$ and $(\sigma,W)$ a smooth $A[H]$-module. Then we write $\operatorname{Ind}_H^G(\sigma)$ (respectively $\operatorname{c-Ind}_H^G(\sigma)$) for the space of functions (functions whose support modulo $H$ is compact) $f\colon G\to W$ satisfying $f(hg)=\sigma(h)f(g)$ for all $h\in H,g\in G$ and which are moreover smooth with respect to the action of $G$ via right translation.\\
For any representation $(\sigma,W)$ of $H$ we define for $g\in G$ the representation $(g\cdot\sigma,gW)$ of $g\cdot H=gHg^{-1}$ on $W$ via $(g\cdot\sigma)(x)w=\sigma(g^{-1}xg)w$ for $x\in g\cdot H,w\in W$. 
For a subgroup $H$ of $G$ and a character $\theta\colon H\to A^\times$ we will denote by $V_{H,\theta}$ the quotient module $V/V(H,\theta)$ where $V(H,\theta)$ is the $A$-submodule generated by all elements $\sigma(h)v-\theta(h)v$ for $h\in H,v\in V$. Moreover, let $V^{H,\theta}$ be the $A$-submodule of $V$ consisting of elements $v$ such that $\sigma(h)v=\theta(h)v$ for all $h\in H$.\\
If we have a $\mathbb Z[1/p]$-algebra homomorphism $\alpha\colon A\to A'$ we will write $(\sigma\otimes_A A',W\otimes_A A')$ for the $A'[H]$-module $W\otimes_A A'$ where $H$ acts on the first factor. If it is clear from the context we often just write $\otimes$ instead of $\otimes_A$. Moreover, for any such algebra homomorphism $\alpha$ we will write $\hat{\alpha}$ for the canonical map from $A[[X]][X^{-1}]$ to $A'[[X]][X^{-1}]$ induced by $\alpha$.

Suppose now that $G$ are the $F$-points of a reductive algebraic group defined over $F$. Then Vigneras (\cite{vigneras1996representations}) constructs an $A$-valued left Haar measure on any group containing an open compact pro-$p$ subgroup. In particular we get an $A$-valued measure on any closed subgroup $H$ of $G$ which we denote by $\mu_H$. For any such $H$ of $G$ we will denote by $\delta_H$ its modulus character. One can show that $\delta_H$ has values in the units of $A$ (it is always a power of $p$) and is multiplicative. %Moreover, for any locally constant compactly supported function $f\colon H\to A$ the equation
%$$\int_Hf(xh)\mu_H(x)=\delta_H(h)\int_Hf(x)\mu_H(x)$$
%holds.\\

If $P=MN$ is a parabolic subgroup of $G$ and $(\sigma,W)$ a smooth representation of $M$ we will write $i_P^G$ for the normalized parabolic induction $\Ind_P^G(\delta_P^{1/2}\otimes_{\mathbb Z[1/\sqrt{p}]}\sigma)$. Moreover, for a  representation $(\pi,V)$ of $G$ we will denote the normalized Jacquet module  $(\delta_P^{-1/2}\otimes_{\mathbb Z[1/\sqrt{p}]}\pi,V_N)$ by $r_P^G(\pi)$.

\subsection{Hecke Algebras}
Let $H$ be any closed subgroup of $G$. Then the Hecke algebra $\mathcal H(H)$ is the $A$-module
$$\{f\colon H\to A\mid f\text{ is locally constant and has compact support}\},$$
where multiplication is given by the convolution product
$$(\alpha *\beta)(h)=\int_H\alpha(x)\beta(x^{-1}h)dx,$$
for $\alpha,\beta\in\mathcal H(H)$. \\
Recall that there is a standard correspondence between smooth left $A[G]$-modules and nondegenerate left $\mathcal H(G)$-modules. We can define a right $\mathcal H(G)$-module structure on a smooth left $A[G]$-module $V$ via 
$$vf=\int_Gf(g)g^{-1}vdg,$$
where $v\in V$ and $f\in\mathcal H(G)$. We then have the following results.
\begin{proposition}
Let $V$ be a smooth left $A[G]$-module, then
$$\mathcal H(G)\otimes_{\mathcal H(G)} V\cong V\otimes_{\mathcal H(G)}\mathcal H(G)\cong V.$$
\end{proposition}

If $H$ is a closed subgroup of $G$ the Hecke algebra $\mathcal H(H)$ is not necessarily a subalgebra of $\mathcal H(G)$. However, one can define a right $\mathcal H(H)$-module structure on $\mathcal H(G)$ via
$$(\alpha\cdot\eta)(g)=\int_H\alpha(gh^{-1})\eta(h)dh,$$
where $\alpha\in\mathcal H(G)$ and $\eta\in\mathcal H(H)$. With this action $\mathcal H(G)$ becomes a $(\mathcal H(G),\mathcal H(H))$-bimodule.\\ 
\begin{proposition}
Let $V$ (respectively $V'$) be a smooth left $A[G]$-module (respectively $A[H]$-module). We have
$$V\otimes_{\mathcal H(G)}(\mathcal H(G)\otimes_{\mathcal H(H)}V')\cong (V\otimes_{\mathcal H(G)}\mathcal H(G))\otimes_{\mathcal H(H)}V'\cong V\otimes_{\mathcal H(H)}V'.$$
\end{proposition}
For a smooth left $A[H]$-module $(V,\pi)$ we can form the tensor product $\mathcal H(G)\otimes_{\mathcal H(H)}V$ and let
$$\Phi\colon\mathcal H(G)\otimes_{\mathcal H(H)}V\to\operatorname{c-Ind}_H^G(V\otimes\delta_H)$$ 
be the map defined by
$$\Phi(f\otimes v)(g)=\int_Hf(g^{-1}h^{-1})\pi(h^{-1})vdh.$$
We have the following result.
\begin{proposition}\label{comphecke} Suppose there exists a neighbourhood basis $\{K_i\}_{i\in I}$ of open compact subgroups of the identity in $G$ such that for all $K_i$, the volume $\mu_H(gK_ig^{-1}\cap H)$ is invertible in $A$ for all $g\in G$. Then the above defined map $\Phi$ is an isomorphism of $\mathcal H(G)$-modules.
\end{proposition}
The above conditions are satisfied for a parabolic subgroup $P$ of $G$. Namely, let $K'$ be an open compact subgroup of $G$ such that $\mu_P(P\cap K')$ is nonzero and invertible (for example $P\cap K'$ is a pro $p$-subgroup of $P$). We have that $\mu_P(p\gamma K\gamma^{-1}p^{-1}\cap P)=\mu_P(p(\gamma K\gamma^{-1}\cap P)p^{-1})=\delta_P(p)\mu_P(\gamma K\gamma^{-1}\cap P)$ for any compact open subgroup $K$ of $G$ and $\gamma\in G$. There is a compact subset $C$ such that $PC=G$, so by the above considerations it is enough to show there is a small enough $K$ such that for all $c\in C$ we have that $cKc^{-1}\subseteq K'$. But now there are finitely many $c_1,\dotsc,c_l$ such that $C\subseteq\bigcup_{i=1}^lK'c_i$ and hence if we choose $K$ to be a compact open subgroup of the open subgroup $\bigcap_{i=1}^lc_i^{-1}K'c_i$ we obtain what we want.\\
Suppose we have two smooth (left) $A[H]$-modules $U,V$. Then we have a diagonal action of $H$ on $U\otimes_AV$. The coinvariants under this action
$$U\otimes_A V/\langle u\otimes v-hu\otimes hv\mid u\in U,v\in V,h\in H\rangle$$
are isomorphic to the tensor product $U\otimes_{A[H]}V$, where $U$ becomes a right $A[H]$-module via the action $u\cdot h=h^{-1}u$ for $h\in H,u\in U$. We will need the following result.
\begin{lemma}\label{tensprod} The equation
	$$\langle h^{-1}u\otimes v-u\otimes hv\mid u\in U,v\in V,h\in H\rangle=\langle uf\otimes v-u\otimes fv\mid u\in U,v\in V,f\in\mathcal H(H)\rangle$$
	holds. In particular we can identify $U\otimes_{A[H]}V$ with $U\otimes_{\mathcal H(H)}V$.
\end{lemma}

\section{Classical Groups}
In this section we describe the classical groups that will be considered in the following sections. Let $W$ be a free left $D$-module of rank $n$. Suppose we have an $F$-bilinear map $h\colon W\times W\to D$. We will distinguish between the nondegenerate and the linear case.

\subsubsection{Nondegenerate Case}
Suppose $D$ has center $E$ (which is a finite extension of $F$) and comes equipped with an involution $\rho$ such that $F$ is the fixed field when $\rho$ is restricted to $E$. Moreover, let $\epsilon\in E$ be either $1$ or $-1$. We assume that $h$ satisfies
$$\rho(h(w,v))=\epsilon h(v,w)$$
and $$h(av,bw)=ah(v,w)\rho(b)$$
for all $v,w\in W$ and $a,b\in D$.
We call $h$ hermitian if $\epsilon=1$ and skew-hermitian if $\epsilon=-1$.
We will consider the following cases for $D,E$ and $\rho$:
\begin{enumerate}
	\item[(I1)] $D=E=F$ and $\rho$ is the identity,
	\item[(I2)] $D$ is a nonsplit quaternion algebra over $E=F$ and $\rho$ is the canonical involution of $D$,
	\item[(I3)] $D=E$ is a quadratic extension of $F$ and $\rho$ is the nontrivial element of $\Gal(E/F)$.
\end{enumerate}
Here we also assume that $h$ is nondegenerate, i.e. that $W^\perp=\{w \in V\mid h(w,w')=0\text{ for all }w'\in W\}$ is trivial. We will refer to this as case (I).
\subsubsection{Linear Case}
Here $D$ is an arbitrary division algebra with center $E=F$ and $h=0$. We will refer to this setting as the linear case or case (II).

Whenever we will speak about a $D$-module $W$ equipped with a bilinear form $h$ we refer to one of the above cases. Let $\End(W,D)$ be the $D$-linear endomorphisms of $W$ (which act on $W$ on the right). The groups we will work with are the isometry groups of a pair $(W,h)$, i.e. 
$$G\coloneqq\operatorname{Isom}(W,h)=\{g\in\End(W,D)^\times\mid h(vg,wg)=h(v,w)\text{ for all }v,w\in W\}.$$
These are the $F$-points of a (possibly disconnected) reductive algebraic group.

\section{The Doubling Method}\label{sectiondoublingzeta}
We will now give a quick introduction to the setup of the doubling method. Assume that $(W,h)$ is of case (I) or (II) as in the above section. Consider the free left $D$-module $W^\Box=W\times W$ together with the $F$-bilinear form
$$h^\Box((v_1,v_2),(w_1,w_2))= h(v_1,w_1)-h(v_2,w_2).$$
Moreover, let $G^\Box$ be the isometry group of the doubled space $(W^\Box,h^\Box)$.
We have a closed topological embedding $$G\times G\hookrightarrow G^\Box$$ where $(w_1,w_2)(g_1,g_2)=(w_1g_1,w_2g_2)$. Let $$W^\bigtriangleup=\{(w,w)\in W^\Box\mid w\in W\}$$
and $$W^\bigtriangledown=\{(w,-w)\in W^\Box\mid w\in W\}.$$
If $h$ is nondegenerate both of the above spaces are maximal totally isotropic subspaces of $(W^\Box,h^\Box)$. Let $P=P_{W^\bigtriangleup}$ be the parabolic subgroup of $G^\Box$ which stabilizes $W^\bigtriangleup$. We set $\overline{P}$ to be the stabilizer of $W^\bigtriangledown$ which is the parabolic opposite to $P$. We have Levi decompositions $P=MN$ and $\overline{P}=M\overline{N}$, where $M$ consists of the elements of $G^\Box$ that stabilize $W^\bigtriangleup$ and $W^\bigtriangledown$. Note that
$$(G\times G)\cap P=\{(g,g)\mid g\in G\}.$$ Let $\Delta\colon P\to E^\times$ be the character defined via
\begin{equation*}
\Delta(x)=\begin{cases}
\operatorname{Nrd}_{W^\bigtriangleup}(x)&\text{in the nondegenerate case,}\\
\operatorname{Nrd}_{W^\bigtriangleup}(x)\operatorname{Nrd}_{W^\Box/W^\bigtriangleup}(x)^{-1}&\text{in the linear case.}
\end{cases}
\end{equation*}

We fix a smooth character $\omega\colon E^\times\to A^\times$ and define a character $\omega_X\colon E^\times\to A[X,X^{-1}]^\times$ via $$\omega_X(e)=\omega(e)X^{\operatorname{val}_E(e)},$$
which defines a smooth representation $\omega_X\circ\Delta$ of $P$ on $A[X,X^{-1}]$.
The induced $G^\Box$-representation
$$I(X,\omega)\coloneqq i_{P}^{G^\Box}(\omega_X\circ\Delta)$$
will be of importance later on. We use normalized induction here, i.e. $I(X,\omega)$ is the set of smooth vectors in 
$$\{f\colon G^\Box\to A[X,X^{-1}]\mid f(sx)=\delta_P(s)^{1/2}\omega_X(\Delta(s))f(x)\text{ for all }s\in P,x\in G^\Box\}.$$
We choose a maximal compact subgroup $K\subseteq G^\Box$ in good position to $P$, i.e. $PK=G^\Box$. This allows us to extend $\operatorname{val}_E(\Delta(\_))$ to a right $K$-invariant function on $G^\Box$ for which we have the following result.
\begin{proposition}\label{delprop}
	The map $g\mapsto\operatorname{val}_E(\Delta(g,1))$ from $G$ to $\mathbb Z$ is locally constant, proper and satisfies $$\operatorname{val}_E(\Delta(g,1))\geq 0$$ for all $g\in G$.
\end{proposition}
\begin{proof}
See Proposition 6.4 in \cite{piatetski1987functions} and Lemma 8.4 in \cite{kakuhama2019local}.
\end{proof}
For each positive integer $N$ we define a function $\alpha_N\colon G\to\mathbb Z$ via
$$\alpha_N(g)=
\begin{cases*}
1&\text{if }$\operatorname{val}_E(\Delta(g,1))\leq N$,\\
0&\text{otherwise.}
\end{cases*}
$$
By the above proposition these maps are locally constant and have compact support.\\
\subsection{The Doubling Zeta Integral}
In this section we will introduce the doubling zeta integral and start analysing it.
Let $(V,\pi)$ be a smooth $A[G]$-module. We will write $M(\pi)$ for the $A$-module of matrix coefficients of $V$, i.e. the $A$-module spanned by functions $\varphi\colon G\to A$ of the form $g\mapsto\lambda(\pi(g)v)$ where $v\in V$ and $\lambda$ is a an element of the contragredient $\widetilde{V}$. We define a smooth action of $G\times G$ on $M(\pi)$ via $$(g_1,g_2)\varphi(g)=\varphi(g_2^{-1}gg_1).$$ Clearly we have a surjective $(G\times G)$-intertwining map $V\otimes\widetilde{V}\to M(\pi)$.\\
Consider for $\varphi\in M(\pi),f\in I(X,\omega)$ and $N\in\mathbb Z_{>0}$ the integral
$$Z_N(X,\varphi,f)\coloneqq\int_G\alpha_N(g)f(g,1)\varphi(g)dg=\sum_{j=-\infty}^\infty a_j(N,\varphi,f)X^j,$$
which yields an element of $A[X,X^{-1}]$. Let $(g,1)=sk$ where $s\in P$ and $k\in K$. We have $f(g,1)=\delta_P(s)^{1/2}\omega(\Delta(s))X^{\operatorname{val}_E(\Delta(g,1))}f(k)$. Since $f$ is locally constant this function admits only finitely many different values on $K$, which together with Proposition \ref{delprop} implies:
\begin{enumerate}
	\item There is an $M\in\mathbb Z$ such that $f(g,1)[X^n]=0$ for all $g\in G$ and $n\leq M$,
	\item Let $M'\in\mathbb Z$. There is an $M''\in\mathbb Z$ (depending on $M'$) such that if $\operatorname{val}_E(\Delta(g,1))>M''$ we have that $f(g,1)[X^n]=0$ for all $n\leq M'$.
\end{enumerate} 
%This implies that for all $N$ the coefficients of $X^n$ of $Z_N(X,\varphi,f)$ are zero for $n$ small enough. Moreover, for fixed $n$ as $N$ grows larger the coefficient of $X^n$ in $Z_N(X,\varphi,f)$ becomes constant. This allows us to define a Laurent series 
%$$Z(X,\varphi,f)=\lim_{N\to\infty}Z_N(X,\varphi,f).$$
This implies that there exists an $M_0\in\mathbb Z$ such that $a_j(N,\varphi,f)=0$ for all $j\leq M_0$. Moreover, for fixed $j$ there is an $N_j\in\mathbb N$ such that $a_j(N,\varphi,f)=a_j(N',\varphi,f)$ for all $N,N'\geq N_j$ and we define $a_j(\varphi,f)\coloneqq a_j(N_j,\varphi,f)$. This allows us to define a Laurent series
$$Z(X,\varphi,f)\coloneqq\sum_{j=M_0}^\infty a_j(\varphi,f)X^j\in A[[X]][X^{-1}].$$
\begin{definition}
	For $\varphi\in M(\pi)$ and $f\in I(X,\omega)$ we call the Laurent series $Z(X,\varphi,f)\in A[[X]][X^{-1}]$ the \emph{doubling zeta integral} of $\varphi$ and $f$.
\end{definition}
The doubling zeta integral is bilinear and hence factors through the tensor product $M(\pi)\otimes_A I(X,\omega)$. Moreover, by acting on the second factor $M(\pi)\otimes I(X,\omega)$ becomes an $A[X,X^{-1}]$-module and the doubling zeta integral is an $A[X,X^{-1}]$-linear map. Note that $(G\times G)$ acts on $M(\pi)\otimes I(X,\omega)$ diagonally and the next result shows that the doubling zeta integral is almost invariant under this action.
\begin{proposition}\label{gequi}
	Let $\kappa\colon G\times G\to A^\times$ be the character 
	$$\kappa(g_1,g_2)=\delta_P(g_2,g_2)^{1/2}\omega(\Delta(g_2,g_2)).$$
	For all $\varphi\in M(\pi),(g_1,g_2)\in G\times G$ and $f\in I(X,\omega)$ we have that
	$$Z(X,(g_1,g_2)\varphi,(g_1,g_2)f)=\kappa(g_1,g_2)Z(X,\varphi,f).$$
\end{proposition}
\begin{proof}
	By definition we have $$Z_N(X,(g_1,g_2)\varphi,(g_1,g_2)f)=\int_G\alpha_N(g)f(gg_1,g_2)\varphi(g_2^{-1}gg_1)dg.$$
	A quick computation shows that $\operatorname{val}_E(\Delta(g,g))=0$ for all $g\in G$, so since $f$ is an element of $I(X,\omega)$ we obtain
	$$f(gg_1,g_2)=\delta_P(g_2,g_2)^{1/2}\omega(\Delta(g_2,g_2))X^{\operatorname{val}_E(\Delta(g_2,g_2))}f(g_2^{-1}gg_1,1)=\kappa(g_1,g_2)f(g_2^{-1}gg_1,1).$$
	The unimodularity of $G$ implies
	$$Z_N(X,(g_1,g_2)\varphi,(g_1,g_2)f)=\kappa(g_1,g_2)\int_G\alpha_N(g_2gg_1^{-1})f(g,1)\varphi(g)dg.$$
	Since the support of the map $g\mapsto\alpha_N(g_2gg_1^{-1})$ equals $g_2^{-1}\operatorname{supp}(\alpha_N)g_1$ and the latter is compact we see that there exists $M$ such that $g_2^{-1}\operatorname{supp}(\alpha_N)g_1\subseteq\operatorname{supp}(\alpha_M)$. Analogously for all $M\in\mathbb Z$ there exists $N\in\mathbb Z$ such that $\operatorname{supp}(\alpha_M)\subseteq g_2^{-1}\operatorname{supp}(\alpha_N)g_1$. These two statements imply the result.
	%These two statements imply that
	%$$\lim_{N\to\infty}\kappa(g_1,g_2)\int_G\alpha_N(g_2gg_1^{-1})f(g,1)\varphi(g)dg=\kappa(g_1,g_2)Z(X,\varphi,f).$$
\end{proof}
\section{Rationality of the doubling zeta integral}\label{sectionrational}
Now that we defined the doubling zeta integral we continue by proving a rationality result. More concretely, in this section we will show the following theorem. 
\begin{theorem}\label{theorat}
Let $(\pi,V)$ be an admissible, $G$-finite $A[G]$-module. Suppose that for all parabolic subgroups $P'$ of $G$ the Jacquet module $r^G_{P'}(V)$ is admissible. Then there exists a Laurent polynomial $Q\in S$ such that $Q\cdot Z(X,\varphi,f)$ is a Laurent polynomial for all $\varphi\in M(\pi)$ and $f\in I(X,\omega)$. In particular, all doubling zeta integrals are elements of $S^{-1}A[X,X^{-1}]$.
\end{theorem}
To prove this theorem we will use the following trivial Lemma.
\begin{lemma}\label{fund}
	Let $M$ be an $A[X,X^{-1}]$-module and let $\zeta\colon M\to A[[X]][X^{-1}]$ be an $A[X,X^{-1}]$-linear map. Moreover, suppose $N$ is an $A[X,X^{-1}]$-submodule of $M$ such that there is a Laurent polynomial $Q_0\in S$ such that $Q_0\cdot\zeta(n)\in A[X,X^{-1}]$ for all $n\in N$. If there is a Laurent polynomial $Q_1\in S$ that annihilates $M/N$ then $Q_0Q_1\cdot\zeta(m)\in A[X,X^{-1}]$ for all $m\in M$.
\end{lemma}

To utilize Lemma \ref{fund} we want to identify elements of $M(\pi)\otimes_A I(X,\omega)$ whose corresponding doubling zeta integral is a Laurent polynomial. By Proposition \ref{gequi} we know that $(g_1,g_2)(\varphi\otimes f)-\kappa(g_1,g_2)\varphi\otimes f$, where $(g_1,g_2)\in G\times G,\varphi\in M(\pi),f\in I(X,\omega)$, lies in the kernel of the doubling zeta integral, which already gives us a big number of the sought after elements. Next, we will analyse the module $I(X,\omega)$ more closely and see that it has a handy filtration.
\subsection{A filtration of $I(X,\omega)$}
We exclude the linear case for now.  Recall that the subspace $W^\bigtriangleup=\{(w,w)\in W^\Box\mid w\in W\}$ is maximal isotropic and $P$ is its stabilizer. Hence Witt's theorem allows us to identify $P\backslash G^\Box$ with the set of maximal isotropic subspaces of $(W^\Box,h^\Box)$. Consider the double quotient $P\backslash G^\Box/(G\times G)$ for which we have the following result.
\begin{lemma} Two maximal isotropic subspaces $U,U'$ of $W^\Box$ lie in the same $(G\times G)$-orbit if and only if $\rk_D(U\cap(W,0))=\rk_D(U'\cap (W,0))$. Moreover, for any maximal isotropic subspace $U$ of $W^\Box$ we have $$\rk_D(U\cap(W,0))=\rk_D(U\cap(0,W)).$$
\end{lemma}\begin{proof}
This is Lemma 2.1 in \cite{piatetski1987functions}.
\end{proof}
Let $r$ be the Witt index of $(W,h)$. For $0\leq j\leq r$ let $U_j$ be a totally isotropic subspace of $(W,h)$ with rank $j$ and set $U_j^\perp=\{w\in W\mid h(w,u)=0\text{ for all }u\in U_j\}$. The maximal isotropic subspace
\begin{equation*}\label{eqtotiso}
\hat{U}_j=\{(u_1,u_2)\in U_j^\perp\times U_j^\perp\mid u_1-u_2\in U_j\}
\end{equation*}
of $(W^\Box,h^\Box)$ satisfies $\rk_D(\hat{U}_j\cap (W,0))=j$. Choose $\delta_j\in G^\Box$ such that  $P\delta_j$ corresponds to $\hat{U}_j$ and set $\Omega_j=P\delta_j(G\times G)$.\\
The orbit of the identity $\Omega_0=P(G\times G)=P(G\times 1)$ is the unique orbit which is open and dense in $G^\Box$. Moreover, we have 
$$\overline{\Omega_j}=\bigcup_{l\geq j}\Omega_l.$$
For $0\leq j\leq r$ set
$$I(X,\omega)^{(j)}=\{f\in I(X,\omega)\mid\operatorname{supp}(f)\subseteq\bigcup_{k\leq j}\Omega_k\}$$ 
which defines a filtration
$$I(X,\omega)=I(X,\omega)^{(r)}\supseteq I(X,\omega)^{(r-1)}\supseteq\dotsc\supseteq I(X,\omega)^{(0)}.$$
The $I(X,\omega)^{(j)}$ are closed under the action of $G\times G$ and $A[X,X^{-1}]$. We will need the following topological lemma.
\begin{lemma}\label{homeo}
	For $0\leq j\leq r$ we have a homeomorphism
	$$P\backslash P\delta_j(G\times G)\cong(\delta_j^{-1}P\delta_j\cap(G\times G))\backslash (G\times G).$$
\end{lemma}
%\begin{proof} The proof of this is standard.
%	We first define an action of $P$ on $P\times (G\times G)$ via $t\cdot (s,u)=(ts,u)$. Moreover $(\delta_j^{-1}P\delta_j\cap(G\times G))$ acts on the right via $(s,u)\cdot v=(s\delta_jv\delta_j^{-1},uv)$ and these actions are compatible. If we take quotients with respect to these actions in different orders we obtain homeomorphic quotients, Firstly, we have
%	$$(P\backslash (P\times (G\times G)))/\operatorname{Stab(\Omega_j)}\cong (G\times G)/\operatorname{Stab}(\Omega_j)$$
%	and on the other hand we need to describe 
%	$$(P\times (G\times G))/\operatorname{Stab}(\Omega_j).$$
%	Let $P\times (G\times G)$ act on the $l$-space $P\delta_j(G\times G)$ via $(s,u)\cdot x=sxu^{-1}$. Then by Corollary 1.6 of \cite{bernstein1976representations} we obtain a homeomorphism between $P\times(G\times G)/S'$, where $S'=\{(\delta_ju\delta_j^{-1},u)\mid u\in \operatorname{Stab}(\Omega_j)\}$ is the stabilisator of $\delta_j$ of the just defined action, and $P\delta_j(G\times G)$ given by the map $$(s,u)\mapsto s\delta_ju^{-1}.$$ Note that $P\times(G\times G)/S'$ is exactly $(P\times (G\times G))/\operatorname{Stab}(\Omega_j)$. Hence we have
%	$$(P\times (G\times G))/\operatorname{Stab}(\Omega_j)=(P\times (G\times G))/S'\cong P\delta_j(G\times G).$$
%	The above homeomorphism commutes with the $P$-action and hence we get
%	$$P\backslash ((P\times (G\times G))/\operatorname{Stab}(\Omega_j))\cong P\backslash P\delta_j(G\times G).$$
%	Inversion now yields the result.
%\end{proof}
\begin{lemma}\label{levelzero}
	Suppose that $f$ is an element of $I(X,\omega)^{(0)}$. Then for all $\varphi\in M(\pi)$ the doubling zeta integral $Z(X,\varphi,f)$ is a Laurent polynomial.
\end{lemma}
\begin{proof}
	%Let $X=\{g\in G^\Box\mid f(g)\not=0\}$ be the support of $f$. Since $f$ is locally constant, $X$ is both open and closed. Let $q\colon G^\Box\to P\backslash G^\Box$ be the quotient map. Note that $q$ is an open map and since $q(X^\complement)=q(X)^\complement$ we obtain that $q(X)$ is a closed subset of the compact space $ P\backslash G^\Box$ and hence also compact. By Lemma \ref{homeo} the canonical map $\varphi\colon P\backslash P(G\times 1)\to G$ is a homeomorphism. Since we assume that $\operatorname{supp}(f)\subseteq P(G\times 1)$ we see that $\varphi(q(X))=\{g\in G\mid f(g,1)\not=0\}$ is a compact subset of $G$. 
	
	One can show that an element $f$ of $I(X,\omega)^{(0)}$ restricted to $G\times 1$ has compact support.
	Hence there is an $M$ such that $\operatorname{supp}(f|_{G\times 1})\subseteq\operatorname{supp}(\alpha_M)$ which implies $Z_N(X,\varphi,f)=Z_M(X,\varphi,f)$ for $N\geq M$ and thus proves the result.
\end{proof}
By using Lemma \ref{fund} we will inductively show that the doubling zeta integrals of elements in $M(\pi)\otimes I(X,\omega)^{(j)}$ lie in $S^{-1}A[X,X^{-1}]$ for $1\leq j\leq r$. For $1\leq j\leq r$ let $Q_j(X,\omega)$ be the $A[G\times G]$-module $I(X,\omega)^{(j)}/I(X,\omega)^{(j-1)}$.

\begin{lemma}
	We have an isomorphism $Q_j(X,\omega)\cong\operatorname{c-Ind}^{G\times G}_{\operatorname{Stab}(\Omega_j)}(\xi_X^{(j)})$ for $1\leq j\leq r$. Here $$\operatorname{Stab}(\Omega_j)=\{x\in (G\times G)\mid P\delta_jx=P\delta_j\}=(G\times G)\cap\delta_j^{-1}P\delta_j$$ and the map $\xi_X^{(j)}\colon\operatorname{Stab}(\Omega_j)\to A[X,X^{-1}]^\times$ is given by
	$$\xi_X^{(j)}(x)=\delta_P(\delta_jx\delta_j^{-1})^{1/2}\omega_X(\Delta(\delta_jx\delta_j^{-1})).$$
	
\end{lemma}
\begin{proof}
	Lemma 6.1.1 of \cite{casselman1995introduction} yields a short exact sequence
	$$0\to I(X,\omega)^{(r-1)}\to I(X,\omega)^{(r)}\to J\to 0$$
	where $J$ is the set of locally constant functions $f\colon \Omega_j\to A[X,X^{-1}]$ such that $f(sx)=\omega_X(\Delta(s))f(x)$ for $s\in P, x\in \Omega_j$ and which are compactly supported modulo $P$. Lemma \ref{homeo} lets us identify $J$ with $\operatorname{c-Ind}^{G\times G}_{\operatorname{Stab}(\Omega_j)}(\xi_X^{(j)})$.
\end{proof}
We give a more detailed description of $\operatorname{Stab}(\Omega_j)$. Recall that $P\delta_j$ corresponds to the maximal isotropic subspace 
$$\{(u_1,u_2)\in U_j^\perp\times U_j^\perp\mid u_1-u_2\in U_j\}$$
of $(W^\Box,h^\Box)$ where $U_j$ is a totally isotropic subspace of $(W,h)$ of rank $j$. For $1\leq j\leq r$ let $P_j=M_jN_j$ be the parabolic subgroup of $G$ that stabilizes $U_j$. Note that $M_j$ is isomorphic to the product $\GL_D(U_j)\times\operatorname{Isom}(U_j^\perp/U_j,h)$.
\begin{proposition}\label{stabprop}
	The stabilisator $\operatorname{Stab}(\Omega_j)$ is contained in $P_j\times P_j$ and $\operatorname{Stab}(\Omega_j)\supseteq N_j\times N_j$. Moreover, $\operatorname{Stab}(\Omega_j)$ contains the $\GL_D(U_j)\times\GL_D(U_j)$ factor of $M_j\times M_j$.
\end{proposition}
\begin{proof}
The first part is Proposition 2.1 in \cite{piatetski1987functions}. The second part follows in a similar way.
\end{proof}

\subsection{Proof of Theorem \ref{theorat}}
Suppose we have proved that there is a Laurent polynomial $Q\in A[X,X^{-1}]$ such that $Q\cdot Z(X,\varphi,f)\in A[X,X^{-1}]$ for all $f\in I(X,\omega)^{(j-1)}$ and $\varphi\in M(\pi)$. We will proceed inductively and use Lemma \ref{fund} to prove that there is a Laurent Polynomial $Q'$ such that $Q'\cdot Z(X,\_)$ is a Laurent polynomial on $M(\pi)\otimes_AI(X,\omega)^{(j)}$. Proposition \ref{levelzero} yields this for $M(\pi)\otimes_A I(X,\omega)^{(0)}$.\\ 
Let $M(\pi)^\kappa$ be the $(G\times G)$-representation $M(\pi)\otimes_A\kappa^{-1}$, where $\kappa$ was defined in Proposition \ref{gequi}. We consider $M(\pi)^\kappa$ as a right $A[G\times G]$-module by setting $\varphi(g_1,g_2)\coloneqq(g_1^{-1},g_2^{-1})\varphi$.\\
Let $S_j,T_j$ be the two submodules of $M(\pi)^\kappa\otimes_AI(X,\omega)^{(j)}$ where $T_j$ is generated by the set $$\{\varphi g\otimes f-\varphi\otimes gf\mid\varphi\in M(\pi)^\kappa,f\in I(X,\omega)^{(j)},g\in G\times G\}$$ and $S_j$ is generated by the elements $\varphi\otimes f$ where $\varphi\in M(\pi)^\kappa$ and $f\in I(X,\omega)^{(j-1)}$. Then by assumption and Proposition \ref{gequi} there is a Laurent polynomial $Q$, such that $Q$ times the doubling zeta integral of elements in $S_j+T_j$ is a Laurent polynomial. Hence to use Lemma \ref{fund} we need to show that there exists a nonzero Laurent polynomial in $S$ that annihilates the $A[X,X^{-1}]$-module
$$(M(\pi)^\kappa\otimes_AI(X,\omega)^{(j)})/(S_j+T_j).$$
\begin{lemma}\label{lem:calcmod}
	We have an $A[X,X^{-1}]$-module isomorphism
	$$(M(\pi)^\kappa\otimes_AI(X,\omega)^{(j)})/(S_j+T_j)\cong M(\pi)^\kappa_{N_j\times N_j}\otimes_{A[M_j\times M_j]}C_j(X,\omega),$$
	where $C_j(X,\omega)\coloneqq\cInd_{\operatorname{Stab}(\Omega_j)}^{P_j\times P_j}(\xi_X^{(j)})\otimes\delta_{P_j\times P_j}^{-1}$.
\end{lemma}
\begin{proof}
	
	We have $$(M(\pi)^\kappa\otimes_AI(X,\omega)^{(j)})/(S_j+T_j)\cong ((M(\pi)^\kappa\otimes_AI(X,\omega)^{(j)})/S_j)/((S_j+T_j)/S_j).$$
	Note that
	$$(M(\pi)^\kappa\otimes_AI(X,\omega)^{(j)})/S_j\cong M(\pi)\otimes_AQ_j(X,\omega)$$
	and the above isomorphism maps $(S_j+T_j)/S_j$ to the submodule generated by the elements 
	$$\varphi\otimes g(f+I(X,\omega)^{(j-1)})-\varphi g\otimes (f+I(X,\omega)^{(j-1)}),$$
	where $\varphi\in M(\pi)^\kappa,f\in I(X,\omega)$ and $g\in G\times G$.
	This proves that
	$$(M(\pi)^\kappa\otimes_AI(X,\omega)^{(j)})/(S_j+T_j)\cong M(\pi)^\kappa\otimes_{A[G\times G]}Q_j(X,\omega).$$
	This isomorphism respects the $A[X,X^{-1}]$-structures and it is easy to check that the same holds for the following morphisms. By Lemma \ref{tensprod} we have
	$$M(\pi)^\kappa\otimes_{A[G\times G]}Q_j(X,\omega)\cong M(\pi)^\kappa\otimes_{\mathcal H(G\times G)}Q_j(X,\omega)$$
	and since compact induction is transitive we obtain
	$$M(\pi)^\kappa\otimes_{\mathcal H(G\times G)}Q_j(X,\omega)\cong M(\pi)^\kappa\otimes_{\mathcal H(G\times G)}\Ind_{P_j\times P_j}^{G\times G}(\operatorname{c-Ind}_{\operatorname{Stab(\Omega_j)}}^{P_j\times P_j}(\xi_X^{(j)})).$$
	Theorem \ref{comphecke} now shows that
	$$\Ind_{P_j\times P_j}^{G\times G}(\operatorname{c-Ind}_{\operatorname{Stab(\Omega_j)}}^{P_j\times P_j}(\xi_X^{(j)}))\cong \mathcal H(G\times G)\otimes_{\mathcal H(P_j\times P_j)}\left((\operatorname{c-Ind}_{\operatorname{Stab(\Omega_j)}}^{P_j\times P_j}(\xi_X^{(j)}))\otimes\delta_{P_j\times P_j}^{-1}\right).$$
	Define $C_j(X,\omega)$ to be $\operatorname{c-Ind}_{\operatorname{Stab(\Omega_j)}}^{P_j\times P_j}(\xi_X^{(j)})\otimes\delta_{P_j\times P_j}^{-1}$. We obtain that 
	\begin{align*}
	M(\pi)^\kappa\otimes_{A[G\times G]}Q_j(X,\omega)&\cong M(\pi)^\kappa\otimes_{\mathcal H(G\times G)}(\mathcal H(G\times G)\otimes_{\mathcal H(P_j\times P_j)}C_j(X,\omega))\\
	&\cong (M(\pi)^\kappa\otimes_{\mathcal H(G\times G)}\mathcal H(G\times G))\otimes_{\mathcal H(P_j\times P_j)}C_j(X,\omega)\\
	&\cong M(\pi)^\kappa\otimes_{\mathcal H(P_j\times P_j)}C_j(X,\omega)\\
	&\cong M(\pi)^\kappa\otimes_{A[P_j\times P_j]}C_j(X,\omega).
	\end{align*}
	Note that $\xi_X^{(j)}$ and $\delta_{P_j\times P_j}^{-1}$ are trivial on $N_j\times N_j$. Hence for $x=m'n'\in P_j\times P_j,n\in N_j\times N_j$ and $f\in C_j(X,\omega)$, Proposition \ref{stabprop} yields
	$$n\cdot f(x)=f(xn)=f(mnn'm^{-1}m)=f(m)=f(mn'm^{-1}m)=f(x).$$
	We have that $M(\pi)^\kappa\otimes_{A[P_j\times P_j]}C_j(X,\omega)$ is isomorphic to 
	$$\left(M(\pi)^\kappa\otimes_{A}C_j(X,\omega)\right)/(U_j+V_j),$$
	where $U_j$ is the $A$-submodule generated by elements $\varphi\otimes f-m\varphi\otimes mf$ where $\varphi\in M(\pi)^\kappa,f\in C_j(X,\omega)$ and $m\in M_j\times M_j$. The $A$-submodule $V_j$ is generated by elements $\varphi\otimes f-n\varphi\otimes nf=(\varphi-n\varphi)\otimes f$ where $\varphi\in M(\pi)^\kappa,f\in C_j(X,\omega)$ and $n\in N_j\times N_j$. We have 
	$$\left(M(\pi)^\kappa\otimes_AC_j(X,\omega)\right)/V_j\cong M(\pi)^\kappa_{N_j\times N_j}\otimes_A C_j(X,\omega)$$
	and $(U_j+V_j)/V_j$ is isomorphic to the submodule of the above generated by 
	$$\overline{\varphi}\otimes f-m\overline{\varphi}\otimes mf$$ 
	where $\overline{\varphi}\in M(\pi)^\kappa_{N_j\times N_j},f\in C_j(X,\omega)$ and $m\in M_j\times M_j$. Hence we conclude that 
	$$M(\pi)^\kappa\otimes_{A[P_j\times P_j]}C_j(X,\omega)\cong M(\pi)^\kappa_{N_j\times N_j}\otimes_{A[M_j\times M_j]}C_j(X,\omega).$$
	Overall we obtain an isomorphism of $A[X,X^{-1}]$-modules
	$$M(\pi)^\kappa\otimes_{A[G\times G]}Q_j(X,\omega)\cong M(\pi)^\kappa_{N_j\times N_j}\otimes_{A[M_j\times M_j]}C_j(X,\omega),$$
	which is what we wanted to prove.
\end{proof}
All that is left now to prove Theorem \ref{theorat} is to find a Laurent polynomial in $S$ that annihilates 
$$M(\pi)^\kappa_{N_j\times N_j}\otimes_{A[M_j\times M_j]}C_j(X,\omega).$$
The next proposition will provide a sufficient condition.
\begin{proposition}\label{sufrat}
	Suppose that $\operatorname{End}_{M_j}(V_{N_j})$ is a finitely generated $A$-module. Then we can find a Laurent polynomial in $S$ which annihilates the $A[X,X^{-1}]$-module $$M(\pi)^\kappa_{N_j\times N_j}\otimes_{A[M_j\times M_j]}C_j(X,\omega).$$ 
\end{proposition}
\begin{proof}
	Recall that $M_j\cong\GL_D(U_j)\times\operatorname{Isom}(U_j^\perp/U_j,h)$ and we will identify these two groups under this isomorphism. Let $z$ be the central element $(\varpi_E\id,\id)$ in $M_j$ where $\varpi_E$ is a uniformizer of the ring of integers of $E$. By Proposition $\ref{stabprop}$ both $z_1=(z,\id)$ and $z_2=(\id,z)$ are elements of $\operatorname{Stab}(\Omega_j)$.\\
	For $f\in C_j(X,\omega),p=mn\in P_j\times P_j$ and $i=1,2$ we have
	$$(z_i\cdot f)(p)=f(mz_i)=\delta_{P_j\times P_j}^{-1}(z_i)\xi_X^{(j)}(z_i)f(m)=\delta_{P_j\times P_j}^{-1}(z_i)\xi_X^{(j)}(z_i)f(p).$$
	A quick calculation gives that $\xi_X^{(j)}(z_1)=\delta_P(\delta_jz_1\delta_j^{-1})^{\frac{1}{2}}\omega(\varpi_E^{rj})X^{rj}$, where $r^2=\dim_E(D)$. Hence for $f\in C_j(X,\omega)$ and $\varphi\in M(\pi)$ we have that
	\begin{equation}\label{poly}\varphi\otimes\left(\delta_{P_j\times P_j}^{-1}(z_1)\delta_P(\delta_jz_1\delta_j^{-1})^{\frac{1}{2}}\omega(\varpi_E^{rj})X^{rj}f\right)=\varphi\otimes z_1f=\varphi\cdot z_1\otimes f
	\end{equation}
	in $M(\pi)^\kappa_{N_j\times N_j}\otimes_{A[M_j\times M_j]}C_j(X,\omega).$ Suppose that $\varphi(g)=\lambda(\pi(g)v)$ where $v\in V,\lambda\in\widetilde{V}$. Then we have $(\varphi\cdot z_1)(g)=\lambda(\pi(gz^{-1})v)$. Since $z^{-1}$ is in the center of $M_j$ it gives rise to an element of $\End_{M_j}(V_{N_j})$. If this algebra is a finitely generated $A$-module there is some nonzero monic polynomial $Q$ such that $Q(z^{-1})\in A[M_j]$ annihilates $V_{N_j}$. Then $Q(z_1)\in A[M_j\times M_j]$ annihilates $M(\pi)^\kappa_{N_j\times N_j}$. By Equation (\ref{poly}) this yields a polynomial in $A[X]$ with leading coefficient a unit that annihilates $M(\pi)^\kappa_{N_j\times N_j}\otimes_{A[M_j\times M_j]}C_j(X,\omega)$. By doing the same procedure for $z_1^{-1}$ and adding the resulting polynomial with the one obtained above we get an element in the multiplicative subset $S$ that annihilates $$M(\pi)^\kappa_{N_j\times N_j}\otimes_{A[M_j\times M_j]}C_j(X,\omega).$$
\end{proof}
By Lemma \ref{lem:calcmod} and Proposition \ref{sufrat} the rationality of the doubling zeta integral follows if $\operatorname{End}_{M_j}(V_{N_j})$ is a finitely generated $A$-module for $1\leq j\leq r$. 
\begin{proposition}
	Let $(\pi,V)$ be an admissible $G$-finite representation over a noetherian ring $A$. Then $\operatorname{End}_{A[G]}(V)$ is a finitely generated $A$-module.
\end{proposition}
\begin{proof}
	By assumption there exist $v_1,\dotsc,v_n$ which generate $V$ as an $A[G]$-module. Let $K$ be a compact open subgroup of $G$ such that $v_i\in V^K$ for all $i=1,\dotsc,n$. Since an element of $\operatorname{End}_{A[G]}(V)$ is completely determined by its values on a generating set and it maps $V^K$ to itself, we have by restriction an embedding $\operatorname{End}_{A[G]}(V)\hookrightarrow\operatorname{End}_A(V^K)$. By admissibility, $V^K$ is finitely generated as an $A$-module and hence $\operatorname{End}_A(V^K)$ is a finitely generated $A$-module.
\end{proof}
\begin{proposition}
	Suppose that $(V,\pi)$ is finitely generated as an $A[G]$-module and let $P=MN$ be a parabolic subgroup of $G$. Then $(V_N,\pi_N)$ is finitely generated as an $M$-module.
\end{proposition}
\begin{proof}
	Suppose that $v_1,\dotsc,v_k$ generate $V$ as an $A[G]$-module. Then there is a compact open subgroup $K$ of $G$ such that $v_i\in V^K$ for all $1\leq i\leq k$. Since $P\backslash G$ is compact $P\backslash G/K$ is finite and let $\{g_1,\dotsc,g_l\}$ be a set of coset representatives. Then the $\pi(g_j)v_i$ for $1\leq i\leq k,1\leq j\leq l$ generate $V$ as a $A[P]$-module. Since $N$ acts trivially on $V_N$ we see that $M$ generates $V_N$ as an $A[M]$-module.
\end{proof}
Thus the Jacquet functor for a parabolic subgroup $P=MN$ of $G$ sends $G$-finite representations to $M$-finite representations. Moreover, the following result by Dat allows us to remove the second condition of Theorem \ref{theorat} for quasi-split classical groups.
\begin{theorem}
	Let $G$ be a connected quasi-split classical group and $A$ a noetherian ring. Then the Jacquet functor preserves admissibility for all parabolic subgroups.
\end{theorem}
\begin{proof}
This is Corollary 1.6 i) in \cite{dat2009finitude}.
\end{proof}
\begin{remark} Note that in our setting not all groups are the $F$-points of connected reductive groups. However, if $G$ is not connected, but the above statement is known for $F$-points of the neutral component $G^0$ of $G$, it also holds for $G$. Namely, if $P=MN$ is a parabolic subgroup of $G$, then $P^0=M^0N$, where $M^0=M\cap G^0$, is a parabolic subgroup of $G^0$. Hence parabolic restrictions $r^G_P(\pi,V)$ and $r^{G^0}_{P^0}(\pi|_{G^0},V)$ have the same underlying set of elements. Since $G^0$ is open in $G$ we have that $M^0$ is open in $M$. For any compact open subgroup $K$ of $M$ we can find an open compact subgroup $K^0$ of $M^0$ that is contained in $K$. The invariants $(V_N)^K$ are then contained in $(V_N)^{K^0}$ which is finitely generated. 
\end{remark}
\begin{remark}
In the linear case the situation is a bit different since the orbit structure of $P\backslash G^\Box/(G\times G)$ is slightly more complicated (see Section 4 of \cite{piatetski1987functions}). Namely, $P\backslash G^\Box$ can be identified with the submodules of $W^\Box$ whose rank as a $D$-module is $n$. Let $W^+\coloneqq\{(w,0)\in W^\Box\mid w\in W\}$ and $W^-\coloneqq\{(0,w)\in W^\Box\mid w\in W\}$. Two rank $n$ submodules $L,L'\subseteq W^\Box$ lie in the same $(P,G\times G)$-double coset if and only if $\operatorname{rk}_D(L\cap W^+)=\operatorname{rk}_D(L'\cap W^+)$ and $\operatorname{rk}_D(L\cap W^-)=\operatorname{rk}_D(L'\cap W^-)$ (and these numbers do not need to be equal). The results in this section carry over to the linear case in particular like in Proposition \ref{stabprop} that the stabilisator in $G\times G$ of a double coset is contained in a product of parabolics of $G$ and contains a $\GL_D$ factor of a Levi which allows one to conclude as in Proposition \ref{sufrat}.
\end{remark}

\section{Intertwining operator}\label{sectionintertwining}
We will now introduce the intertwining operator that is needed to formulate the functional equation satisfied by the doubling zeta integrals. We follow the exposition of Waldspurger in Chapter 4 of \cite{waldspurger2003formule}. A similar approach can be found in \cite{dat2005nu}.\\
 We assume for now that we are in the nondegenerate case. Let $A_0$ be a maximal split torus of $G^\Box$ which is contained in $M$. %$A_0\subseteq M$. 
Then we have that $P=MN$ where $M\cong\GL_D(W)$. Let $A_M$ be the largest split torus in the center of $M$, which under the above isomorphism are the diagonal scalar matrices with entries in $E$. 
Note that $\overline{P}$, the opposite parabolic of $P$, satisfies $w_0Pw_0=\overline{P}$ where $w_0=(1,-1)\in G\times G\subseteq G^\Box$.
We choose a left Haar measure on $\overline{N}$ (which via the isomorphism $N\to\overline{N},n\mapsto w_0^{-1}nw_0$ defines a left Haar measure on $N$). Everything in this section depends on the choice of this measure.

We want to construct a nontrivial element of the space $$\Hom_{G^\Box}\left(i_{P}^{G^\Box}(\omega_X\circ\Delta),i_{\overline{P}}^{G^\Box}(\omega_X\circ\Delta)\right),$$
which by Frobenius reciprocity is equivalent to constructing a map in
$$\Hom_M\left(r_{\overline{P}}^{G^\Box}i_{P}^{G^\Box}(\omega_X\circ\Delta),\omega_X\circ\Delta\right).$$
Let $\Lambda$ be the $G^\Box$-representation $r_{\overline{P}}^{G^\Box}i_P^{G^\Box}(\omega_X\circ\Delta)$. Moreover let $^PW^{\overline{P}}\subseteq W_G$ be a set of representatives for $W_M\backslash W_G/W_M$, where $W_H=\operatorname{Norm}_H(A_0)/Z_H(A_0)$ for any subgroup $H$ of $G^\Box$. By the geometric Lemma, $\Lambda$ has a filtration $(\mathcal F_{w})$ where $w\in ^PW^{\overline{P}}$. Namely, we have the Bruhat-Tits decomposition 
$$G^\Box=\bigcup_{w\in ^PW^{\overline{P}}}Pw\overline{P}$$
and we can choose an ordering $1\leq w_1\leq w_2\leq\dotsc \leq w_k$ on $^PW^{\overline{P}}$ such that $\bigcup_{i<l}Pw_i\overline{P}$ is open in $\bigcup_{i\leq l}Pw_i\overline{P}$. In particular $P\overline{P}$ is open in $G^\Box$. Then $\mathcal F_w$ consists of the image in $\Lambda$ of those functions whose support is contained in $\bigcup_{w'\leq w}Pw'\overline{P}$. The quotients $\mathcal F_{w_i}/\mathcal F_{w_{i-1}}$ are isomorphic to 
$$J_{w_i}\coloneqq i^M_{M\cap w_iPw_i^{-1}}\left(w_i\cdot r^M_{M\cap w_i^{-1}\overline{P}w_i}(\omega_X\circ\Delta)\right).$$
%For all $w\in ^PW^{\overline{P}}\backslash\{1\}$ note that $r^M_{M\cap w^{-1}\overline{P}w}(\omega_X\circ\Delta)$ is a representation of $M\cap w^{-1}Mw$ and $w\cdot r^M_{M\cap w^{-1}\overline{P}w}(\omega_X\circ\Delta)$ is then a representation of $wMw^{-1}\cap M$. (It is known that $M\cap w\overline{P}w^{-1}$ is a parabolic subgroup of $M$ with Levi decomposition $M\cap w\overline{P}w^{-1}=(M\cap wMw^{-1})(M\cap w\overline{N}w^{-1})$)\\
%Let $\theta_w\colon A_M\to A[X,X^{-1}]$ be defined via
%$$\theta_w(a)=\delta_{M\cap wPw^{-1}}(a)^{1/2}\delta_{M\cap w^{-1}\overline{P}w}(w^{-1}aw)^{-1/2}\omega(\operatorname{Nrd}_W(w^{-1}aw))X^{\operatorname{val}_E(\operatorname{Nrd}_W(w^{-1}aw))}.$$
%The torus $A_M$ acts on $J_w$ via 
%\begin{equation}\label{eq:acquo}a\cdot f=\delta_{M\cap wPw^{-1}}(a)^{1/2}\delta_{M\cap w^{-1}\overline{P}w}(w^{-1}aw)^{-1/2}\omega(\operatorname{Nrd}_W(w^{-1}aw))X^{\operatorname{val}_E(\operatorname{Nrd}_W(w^{-1}aw))}f,\end{equation}
%\begin{equation}\label{eq:acquo}a\cdot f=\theta_w(a)f,\end{equation}
%where $a\in A_M,f\in J_w$. By the above equation we see that $$a-(w\cdot(\omega_X\circ\Delta))(a)=a-\omega(\operatorname{Nrd}_W(w^{-1}aw))X^{\operatorname{val}_E(\operatorname{Nrd}_W(w^{-1}aw))}$$ annihilates $J_w$ for all $a\in A_M$.\\
%By the proof of Theorem IV.1.1 in \cite{waldspurger2003formule} we have that for $w\in^PW^{\overline{P}}$, where $w\not=1$, there exists an $a_w\in A_M$ such that \begin{equation}\label{waldint}
%X^{\operatorname{val}_E(\operatorname{Nrd}_W(w^{-1}a_ww))}\not=X^{\operatorname{val}_E(\operatorname{Nrd}_W(a_w))}.
The bottom element $J_1$ in the filtration of $\Lambda$ consists of the image in $\Lambda$ of those functions that are supported in $P\overline{P}$. Let $\theta_w\colon A_M\to A[X,X^{-1}]$ be defined via
$$\theta_w(a)\coloneqq\delta_{M\cap wPw^{-1}}(a)^{1/2}\delta_{M\cap w^{-1}\overline{P}w}(w^{-1}aw)^{-1/2}\omega(\operatorname{Nrd}_W(w^{-1}aw))X^{\operatorname{val}_E(\operatorname{Nrd}_W(w^{-1}aw))}.$$
The torus $A_M$ acts on $J_w$ via 
\begin{equation}\label{eq:acquo}a\cdot f=\theta_w(a)f,\end{equation}
where $a\in A_M,f\in J_w$, which shows that $a-\theta_w(a)$ annihilates $J_w$ for all $a\in A_M$.\\
By the proof of Theorem IV.1.1 in \cite{waldspurger2003formule} we have that for $w\in^PW^{\overline{P}}$, where $w\not=1$, there exists an $a_w\in A_M$ such that \begin{equation}\label{waldint}
X^{\operatorname{val}_E(\operatorname{Nrd}_W(w^{-1}a_ww))}\not=X^{\operatorname{val}_E(\operatorname{Nrd}_W(a_w))}.
\end{equation}
%\textcolor{blue}{
%	Suppose that \begin{equation}\label{test}
%	X^{\operatorname{val}(\operatorname{Nrd}_W(w^{-1}aw))}=X^{\operatorname{val}(\operatorname{Nrd}_W(a))}
%	\end{equation}
%	for all $a\in A_M$. Let $\operatorname{Rat}(M)$ be the group of algebraic characters of $M$ defined over $F$. (Since $M$ is isomorphic to $\GL_D(W)$, we obtain that $\operatorname{Rat}(M)$ is generated by the reduced norm) We set $M^1=\bigcap_{\chi\in\operatorname{Rat}(M)}\ker(\abs{\chi}_F)$. Hence if Equation (\ref{test}) is true we must have that $w^{-1}awa^{-1}\in M^1$ for all $a\in A_M.$ We set $$a_0=(\operatorname{Rat}(A_0)\otimes_\mathbb Z\mathbb R)^*$$ and $$a_M=(\operatorname{Rat}(A_M)\otimes_\mathbb Z\mathbb R)^*$$ which yields a canonical decomposition
%	$$a_0=a_M\oplus a^M.$$ Note that $W_G$ acts on $a_0$. Now $w^{-1}awa^{-1}\in M^1$ implies that $$w^{-1}H-H\in a^M$$ for all $H\in a_M$. For all $H\in a_0$ let $H_M$ and $H^M$ be the corresponding components in $a_M$ respectively $a^M$. We obtain that $(w^{-1}H)_M=H$ for $H\in a_M$. Then, 
%	$$\abs{H}^2=\abs{w^{-1}H}^2=\abs{H}^2+\abs{(w^{-1}H)^M}^2$$ which shows that $(w^{-1}H)^M=0$. Hence $w^{-1}H=H$ for all $H\in a_M$, which shows that $w\in W^M$.}\\ 
Note that $a_w-\theta_w(a_w)$ acts on $J_1$ via 
$$\omega(\operatorname{Nrd}_W(a_w))X^{\operatorname{val}_E(\operatorname{Nrd}_W(a_w))}-\theta_w(a_w),$$
which is nonzero by Equation (\ref{waldint}). We define
$$R=\prod_{w\in^PW^{\overline{P}}\backslash \{1\}}\left(a_w-\theta_w(a_w)\right).$$
By construction if $R$ acts on $\Lambda$ it maps every element of $\Lambda$ into $J_1$. On $J_1$ it acts via multiplication by
$$\widehat{R}=\prod_{w\in^PW^{\overline{P}}\backslash \{1\}}\left(\omega(\operatorname{Nrd}_W(a_w))X^{\operatorname{val}_E(\operatorname{Nrd}_W(a_w))}-\theta_w(a_w)\right),$$
which is an element of the multiplicative subset $S\subseteq A[X,X^{-1}]$. Since $A_M$ is a subset of the center of $M$, acting with $R$ yields an element of $\End_M(\Lambda)$.\\
It is straightforward to see that if the support of $f\in I(X,\omega)$ is contained in $P\overline{P}$, its restriction to $\overline{N}$ has compact support and the integral $\int_{\overline{N}}f(\overline{n})d\overline{n}$ is well defined. This defines an $M$-intertwining map which factors through $J_1$ and we obtain 
\begin{align*}
\Psi\colon J_1&\to (\omega_X\circ\Delta)\\
f&\mapsto\int_{\overline{N}}f(\overline{n})d\overline{n}.
\end{align*}
In conclusion, we obtain a map $\Psi\circ R$ which is an element of
$$\Hom_M\left(r_{\overline{P}}^{G^\Box}i_{P}^{G^\Box}(\omega_X\circ\Delta),\omega_X\circ\Delta\right).$$
Frobenius reciprocity yields a map in
$$\Hom_{G^\Box}\left(i_{P}^{G^\Box}(\omega_X\circ\Delta),i_{\overline{P}}^{G^\Box}(\omega_X\circ\Delta)\right)$$
that is characterized by 
$$f\mapsto(g\mapsto\Psi(R(gf))).$$
Let $\overline{\omega}_X\colon E^\times\to A[X,X^{-1}]$ be defined via $\overline{\omega}_X(e)=\omega(\rho(e))X^{\operatorname{val}_E(e)}$ for $e\in E^\times$. There is a $G^\Box$-intertwining map $i_{\overline{P}}^{G^\Box}(\omega_X\circ\Delta)\to i_P^{G^\Box}(\overline{\omega}_X^{-1}\circ\Delta)$ that is given by $$f\mapsto(g\mapsto f(w_0g)).$$
This follows since $\Delta(w_0mw_0)=\rho(\Delta(m^{-1}))$ for all $m\in M$. Hence we obtain an element $\overline{M}_X$ of 
$$\Hom_{G^\Box}\left(i_{P}^{G^\Box}(\omega_X\circ\Delta),i_{P}^{G^\Box}(\overline{\omega}_X^{-1}\circ\Delta)\right).$$
The map $M_X\colon I(X,\omega)\to S^{-1}A[X,X^{-1}]\otimes I(X,\overline{\omega}_X^{-1})$ that interpolates the well known intertwining operator is then given by
$$M_X(f)\coloneqq\widehat{R}^{-1}\overline{M}_X(f).$$

\begin{remark}
In the linear case we find ourselves in a slightly different situation. Here we have that $G=GL_D(W)$ where $D$ is some central division algebra over $F$ and $W$ is a free left $D$-module of rank $n$. Then $M$ is isomorphic to $\GL_D(W^\bigtriangleup)\times GL_D(W\times W/W^\bigtriangleup)$. Hence the maximal split torus $A_M$ in the center of $M$ consists of two copies of $F^\times$ and the argument in \cite{waldspurger2003formule} to find an element that satisfies (\ref{waldint}) does not carry over directly. However, we can see in a more direct way that for all $w\in^{P}W^{\overline{P}},w\not=1$ we can find an $a_w\in A_M$ such that
$$X^{\operatorname{val}_F(\Delta(a_w))}\not=X^{\operatorname{val}_F(\Delta(wa_ww^{-1}))}.$$
Namely, let $a=(\id,\varpi_F\id)\in A_M$ and hence $\operatorname{val}_F(\Delta(a))=-rn$, where $r^2=\dim_F(D)$. Then for all $w\in {^PW^{\overline{P}}},w\not=1$ we have that $waw^{-1}=(a^1_w,a^2_w)\in \GL_D(W^\bigtriangleup)\times GL_D(W\times W/W^\bigtriangleup)$, where both $a^1_w$ and $a^2_w$ are diagonal matrices whose entries are a permutation of those of $a$. Moreover, $a_1^w$ has some diagonal entries which are $\varpi_F$ instead of $1$ which implies that $\operatorname{val}_F(\Delta(waw^{-1}))>-rn$.
\end{remark}

\begin{remark}
One could also define an intertwining operator via $M_X'(f)(g)=\int_{N}f(w_0ng)dn$, which yields a well-defined Laurent series by Lemma II.3.4 of \cite{waldspurger2003formule}. The same argument as in the proof of Theorem IV.1.1. in \emph{op. cit.} shows that $M_X'=M_X$.
\end{remark}
\section{Functional Equation}\label{sectionfunctional}
In this section we will prove the following functional equation.
\begin{theorem}\label{functionaltheorem} Let $(\pi,V)$ be a smooth admissible, $G$-finite $A[G]$-module which satisfies the conditions in Theorem \ref{theorat}. Moreover, we assume that the canonical trace map $\Phi\colon V\otimes_A\widetilde{V}\to A$ is surjective and that
	$$\Hom_{G}(V\otimes\widetilde{V},A)\cong A.$$
	Then there is an unique $\Gamma(X,\pi,\omega)\in S^{-1}A[X,X^{-1}]$ such that 
	$$Z(X^{-1},\varphi,M_X(f))=\Gamma(X,\pi,\omega) Z(X,\varphi,f)$$
	for all $\varphi\in V\otimes\widetilde{V}$ and $f\in I(X,\omega)$. 
\end{theorem}
The proof of this result will take up the rest of this section. Note that $\Gamma(X,\pi,\omega)$ depends on the choice of a Haar measure for $N$ in the definition of the intertwining operator $M_X$. Moreover, in this section we will often view the doubling zeta integral as a function on $(V\otimes\widetilde{V})\otimes I(X,\omega)$ in the straigtforward way. From now on we assume that $(\pi,V)$ satisfies the conditions of the above theorem. First we will show that we can choose input data such that the doubling zeta integral equals one. 
\begin{proposition}\label{exone}
	There are $\varphi\in M(\pi)$ and $f\in I(X,\omega)$ such that $Z(X,\varphi,f)=1$.
\end{proposition}
\begin{proof} We follow the proof of Theorem 3.1 in \cite{rallis2005stability}. Let $\varphi=\sum_{i=1}^nv_i\otimes\lambda_i$ be an element of $V\otimes_A\widetilde{V}$ such that $\Phi(\varphi)=1$. We choose an compact open subgroup $K$ of $G$ such that $v_i\in V^K$ for $1\leq i\leq n$ and such that the volume $\mu_G(K)$ of $K$ is invertible in $A$. Let $f_0\in I(X,\omega)$ be the function which is supported in $P(G\times 1)$ and when restricted to $(G\times 1)$ equals the characteristic function of $(K\times 1)$. Then for all $N\geq 1$ we have that
	
	\begin{align*}
	Z_N\left(X,\varphi,\frac{f_0}{\mu_G(K)}\right)&=\mu_G(K)^{-1}\int_G\alpha_N(g)f_0(g,1)\sum_{i=1}^n\lambda_i(\pi(g)v_i)dg\\
	&=\mu_G(K)^{-1}\int_K\sum_{i=1}^n\lambda_i(\pi(g)v_i)dg,
	\end{align*}
	which equals one and hence implies the result.
\end{proof}
Let $I(X^{-1},\overline{\omega}^{-1})\coloneqq i_P^{G^\Box}(\overline{\omega}_X^{-1}\circ\Delta)$.
We will denote the space $(V\otimes\widetilde{V}\otimes\kappa^{-1})\otimes I(X,\omega)$ by $I$ and by $I^{(0)}$ the subspace $(V\otimes\widetilde{V}\otimes\kappa^{-1})\otimes I(X,\omega)^{(0)}$ of $I$. Let $T$ (respectively $T^{(0)}$) be the subspaces of $I$ generated by ($\varphi\otimes gf-\varphi g\otimes f$) where $\varphi\in V\otimes\widetilde{V}\otimes\kappa^{-1},g\in G\times G$ and $f\in I(X,\omega)$ (respectively $f\in I(X,\omega)^{(0)}$). We denote the space $(V\otimes\widetilde{V}\otimes\kappa^{-1})\otimes I(X^{-1},\overline{\omega}^{-1})$ by $\hat{I}$ and by $\hat{I}^{(0)},\hat{T}$ and $\hat{T}^{(0)}$ the subspaces of $\hat{I}$ with analogous definitions as above for $I$. 
\begin{proposition}\label{exgamma}
	Suppose the representation $(\pi,V)$ satisfies $$\Hom_{A[X,X^{-1}]}(I^{(0)}/T^{(0)},A[X,X^{-1}])\cong A[X,X^{-1}].$$ 
	Then there is a gamma factor $\Gamma(X,\pi,\omega)\in S^{-1}A[X,X^{-1}]$ such that 
	$$Z(X^{-1},\varphi,M_X(f))=\Gamma(X,\pi,\omega) Z(X,\varphi,f)$$
	for all $\varphi\in V\otimes\widetilde{V}$ and $f\in I(X,\omega)$.
\end{proposition}
\begin{proof}
	In the above section we constructed an element $\overline{M}_X$ of
	$$\Hom_{G^\Box}\left(I(X,\omega),I(X^{-1},\overline{\omega}^{-1})\right),$$
	which clearly induces a map from $I$ to $\hat{I}$ and hence also a map
	$$\hat{M}_X\colon I\to\hat{I}/\hat{T}.$$
	Since $\hat{M}_X$ is a $G^\Box$-intertwining it factors through $T$. From Section \ref{sectionrational} we know that there is a polynomial $\widehat{Q}$ such that for all $\alpha\in\hat{I}$ we have that $\widehat{Q}\alpha\in\hat{I}^{(0)}+\hat{T}$. The doubling zeta integral is a well defined map from $(\hat{I}^{(0)}+\hat{T})/\hat{T}$ to $A[X,X^{-1}]$. Overall we obtain a map $\widehat{Z}$ which is defined as the composition
	$$ I/T\xrightarrow[]{\hat{M}_X}\hat{I}/\hat{T}\xrightarrow[]{\widehat{Q}}(\hat{I}^{(0)}+\hat{T})/\hat{T}\xrightarrow[]{Z(X^{-1},\_)}A[X,X^{-1}].$$
	Since we have a canonical map $I^{(0)}/T^{(0)}\to I/T$ we obtain an element of $$\Hom_{A[X,X^{-1}]}(I^{(0)}/T^{(0)},A[X,X^{-1}]),$$
	which by abuse of notation we will also denote by $\widehat{Z}$.
	Clearly, $Z(X,\_)$ is also an element of the above space.
	By Proposition \ref{exone} there is an element $\alpha_0\in I^{(0)}$ such that $Z(X,\alpha_0)=1$ and since we assume that $$\Hom_{A[X,X^{-1}]}(I^{(0)}/T^{(0)},A[X,X^{-1}])\cong A[X,X^{-1}]$$ we obtain that there is an element $\widehat{\Gamma}\in A[X,X^{-1}]$ such that for all $\varphi\otimes f\in I^{(0)}$ we have
	$$\widehat{Z}(\varphi\otimes f)=\widehat{\Gamma} Z(X,\varphi\otimes f).$$
	We now show that if the above equation holds on $I^{(0)}$ then it also holds on $I$. There is a Laurent polynomial $Q\in S$ such that for any $\alpha\in I$ we have $Q\alpha=\alpha^{(0)}+t$ where $\alpha^{(0)}\in I^{(0)}$ and $t\in T$. We obtain
	\begin{align*}
	Q\widehat{\Gamma}\cdot Z(X,\alpha)=\widehat{\Gamma}Z(X,Q\alpha)&=\widehat{\Gamma}Z(X,\alpha^{(0)})\\
	&=\widehat{Z}(\alpha^{(0)})\\
	&=Z(X^{-1},\widehat{Q}\hat{M}_X(Q\alpha-t))\\
	&=QZ(X^{-1},\widehat{Q}\hat{M}_X(\alpha))\\
	&=Q\widehat{Z}(\alpha),
	\end{align*}
	which implies the equation
	\begin{equation}
	\widehat{Z}(\varphi,f)=\widehat{\Gamma}Z(X,\varphi,f)
	\end{equation}
	in $S^{-1}A[X,X^{-1}]$ for all $\varphi\in V\otimes\widetilde{V}$ and $f\in I(X,\omega)$. If we multiply this equation by $\widehat{Q}^{-1}$ and by $\widehat{R}^{-1}$ (which was defined in the last section on intertwining operators) we obtain that
	$$Z(X^{-1},\varphi,M_X(f))=\widehat{Q}^{-1}\widehat{R}^{-1}\widehat{\Gamma}Z(X,\varphi,f)$$
	for all $\varphi\in V\otimes\widetilde{V}$ and $f\in I(X,\omega)$. We set
	$$\Gamma(X,\pi,\omega)\coloneqq\widehat{Q}^{-1}\widehat{R}^{-1}\widehat{\Gamma}\in S^{-1}A[X,X^{-1}]$$
	and the result follows.
	
\end{proof}

%Note that $I(X,\omega)^{(0)}$ is isomorphic as an $A[X,X^{-1}](G\times G)$-module to the locally constant compactly supported $A[X,X^{-1}]$-valued functions on $G$ where the action of $G\times G$ is given via 
%$$(g_1,g_2)f(g)=f(g_2^{-1}gg_1).$$
Note that $I(X,\omega)^{(0)}$ can be identified as an $A[X,X^{-1}](G\times G)$-module with $$S(G,\kappa)\coloneqq\operatorname{c-Ind}_G^{G\times G}(\kappa\otimes_AA[X,X^{-1}])$$ where $G$ is embedded diagonally into $G\times G$ and acts trivially on $A[X,X^{-1}]$.
\begin{lemma}
	We have that
	$$\Hom_{A[X,X^{-1}]}(I^{(0)}/T^{(0)},A[X,X^{-1}])\cong\Hom_{A[G]}(V\otimes\widetilde{V},A[X,X^{-1}]),$$
	where $G$ acts diagonally on $V\otimes\widetilde{V}$.
\end{lemma}

\begin{proof}
	Note that
	$$\Hom_{A[X,X^{-1}]}\left((V\otimes_A\widetilde V\otimes_A\kappa^{-1})\otimes_{A[G\times G]}S(G,\kappa),A[X,X^{-1}]\right)\cong\Hom_{A[G\times G]}\left(V\otimes_A\widetilde V\otimes_A\kappa^{-1},\widetilde{S(G,\kappa)}\right).$$
By Chapter I., Section 5.11 of \cite{vigneras1996representations} we see that %$\widetilde{S(G,\kappa)}=\Hom_{A[X,X^{-1}]}(S(G,\kappa),A[X,X^{-1}])^\infty$ is isomorphic to $\Ind_G^{G\times G}(\kappa^{-1}\otimes_AA[X,X^{-1}])$. We obtain
$\widetilde{S(G,\kappa)}$ is isomorphic to $\Ind_G^{G\times G}(\kappa^{-1}\otimes_AA[X,X^{-1}])$. We obtain
	\begin{align*}
	\Hom_{A[G\times G]}\left(V\otimes_A\widetilde V\otimes_A\kappa^{-1},\widetilde{S(G,\kappa)}\right)&\cong\Hom_{A[G\times G]}\left(V\otimes_A\widetilde{V}\otimes_A\kappa^{-1},\Ind_G^{G\times G}(\kappa^{-1}\otimes_AA[X,X^{-1}])\right)\\
	&\cong\Hom_{A[G]}\left(V\otimes_A\widetilde{V}\otimes_A\kappa^{-1},\kappa^{-1}\otimes_AA[X,X^{-1}]\right)\\
	&\cong\Hom_{A[G]}\left(V\otimes_A\widetilde{V},A[X,X^{-1}]\right)
	\end{align*}
	via Frobenius reciprocity.
\end{proof}
Hence the functional equation for the doubling zeta integral for the representation $(\pi,V)$ follows if
\begin{equation}\label{funceqequi}
\Hom_{G}(V\otimes\widetilde{V},A[X,X^{-1}])\cong A[X,X^{-1}].
\end{equation}
Recall that we assume that the canonical trace map $\Phi\colon V\otimes_A\widetilde{V}\to A$ is surjective. We will show that in this case Equation (\ref{funceqequi}) follows from $\Hom_{A[G]}(V\otimes_A\widetilde{V},A)\cong A$, i.e. every map in this space is given by
$$w\otimes\lambda\mapsto c\lambda(w),$$
for some constant $c\in A$. Indeed, for every $i\in\mathbb Z$ we have the $A$-linear coefficient map $\alpha_i\colon A[X,X^{-1}]\to A$ which sends each Laurent polynomial to its $i$th coefficient. Hence for $$\phi\in\Hom_{A[G]}\left((V\otimes_A\widetilde{V}),A[X,X^{-1}]\right)$$ and $i\in\mathbb Z$ we can define a map $\phi_i\coloneqq\alpha_i\circ\phi$ which is an element of $\Hom_{A[G]}(V\otimes_A\widetilde{V},A)$ and hence given by $$\phi_i(v\otimes\lambda)=c_i\lambda(v)$$
for all $v\in V,\lambda\in\widetilde{V}$ and for some $c_i\in A$. Since $\phi=\sum_{i=-\infty}^{\infty}\phi_iX^i$ this proves that $$\phi(v\otimes\lambda)=\left(\sum_{i=-\infty}^\infty c_iX^i\right)\lambda(v)$$
for all $v\in V,\lambda\in\widetilde{V}$. Now since we assume that the canonical trace map $V\otimes_A\widetilde{V}\to A$ is surjective we obtain that only finitely many of the $c_i$ are nonzero which proves that 
$$\Hom_{A[G]}\left((V\otimes_A\widetilde{V}),A[X,X^{-1}]\right)\cong A[X,X^{-1}]$$
and concludes our proof of Theorem \ref{functionaltheorem}.
\begin{definition}
Let $(\pi,V)$ be an admissible, $G$-finite representation of $G$, such that the assumptions in Theorem \ref{functionaltheorem} hold. We denote by $\Gamma(X,\pi,\omega)\in S^{-1}A[X,X^{-1}]$ (for a fixed measure on $N$) the \emph{unnormalized gamma factor} associated to $(\pi,V)$ and the character $\omega$.
\end{definition}
%Note that $\Gamma(X,\pi,\omega)$ depends on the choice of a measure on $N$ that we made in the definition of $M_X$.
\begin{example}
	Suppose that $A$ is a field. Then any admissible representation is reflexive, i.e. $\widetilde{\widetilde{V}}\cong V$. We obtain
	$$\Hom_G(V\otimes\widetilde{V},A)\cong \Hom_G(V,\widetilde{\widetilde{V}})\cong\Hom_G(V,V)$$
	and hence in this case the functional equation holds for all admissible, $G$-finite representations that satisfy Schur's Lemma.
\end{example}
\subsection{Gamma factor and tensor products}
Let $A'$ be a noetherian $\mathbb Z[1/\sqrt{p}]$-algebra and suppose we have a $\mathbb Z[1/\sqrt{p}]$-algebra homomorphism $\alpha\colon A\to A'$. Moreover, we assume that $V\otimes_AA'$ satisfies the conditions of Theorems \ref{theorat} and \ref{functionaltheorem}. Let $\hat{\alpha}\colon A[[X]][X^{-1}]\to A'[[X]][X^{-1}]$ be the canonical map induced by $\alpha$.
\begin{proposition} We have that
	$$\hat{\alpha}(Z(X,\varphi,f))=Z(X,\alpha\circ\varphi,\alpha\circ f)$$
	and
$$\hat{\alpha}(\Gamma(X,\pi,\omega))=\Gamma(X,\pi\otimes A',\omega\otimes A'),$$
i.e. the gamma factor we just defined behaves well under base-change.
\end{proposition}

In the next subsection we will prove the functional equation using the above result for a class of representations that will play a role in the local Langlands correspondence for families for quasi-split classical groups.
\subsection{residually Co-Whittaker modules}
From now on we will assume that $A$ is a $W(k)$-algebra where $W(k)$ is the ring of Witt vectors of some algebraically closed field of characteristic $\ell\not=p$. Moreover, for this section we will only consider $G$ which are unramified. Hence we can choose a hyperspecial subgroup $K_x$ of $G$ with reductive quotient $G_x$. Let $B_x$ be a Borel subgroup of $G_x$ with unipotent radical $U_x$ and $\psi\colon U_x\to W(k)^\times$ a generic character. We set $K'$ to be the preimage of $U_x$ under the projection $K_x\to G_x$ and thus can view $\psi$ as a character of $K'$. Note that $K'$ is an open compact pro-$p$ subgroup of $G$. For the rest of this section we fix a triple $(K_x,U_x,\psi)$ as above. Let $W$ be the $W(k)[G]$-module $\cInd_{K'}^G(\psi)$. We have the following theorem whose proof will appear in forthcoming work of Dat, Helm, Kurinczuk and Moss.
\begin{theorem}\label{psigentheo}
The ring $E=\operatorname{End}_{W(k)[G]}(W)$ is a reduced, commutative, flat, finitely generated $W(k)$-algebra, and $W$ is an admissible $E[G]$-module. Moreover, $W$ is projective as a $W(k)[G]$-module.
\end{theorem}

\begin{definition}
For a $W(k)$-algebra $A$ we call an $A[G]$-module $(V,\pi)$ residually co-Whittaker of type $(K_x,U_x,\psi)$, if:
\begin{enumerate}
	\item $(V,\pi)$ is an admissible $A[G]$-module,
	\item The map $W\otimes_{W(k)}\Hom_{W(k)[G]}(W,V)\to V$ is surjective, and
	\item $\Hom_{W(k)[G]}(W,V)$ is free of rank one over $A$.
\end{enumerate}
\end{definition}
Theorem \ref{psigentheo} implies that in particular $W$ is a residually co-Whittaker $E[G]$-module of type $(K_x,U_x,\psi)$. %Moreover, residually co-Whittaker modules satisfy many favourable properties like the class of co-Whittaker modules introduced in \cite{helm2016whittaker}. 
Residually co-Whittaker modules are a depth-zero analog to co-Whittaker modules (as introduced in \cite{helm2016whittaker}) and should play a similar role in the Local Langlands conjecture in families for classical groups as co-Whittaker modules played in the Local Langlands correspondence in families for $\GL_n$. The class of residually co-Whittaker is stable under base change and satisfies Schur's Lemma.\\
The module $W$ is a universal object in the class of residually co-Whittaker modules of type $(K_x,U_x,\psi)$. Namely, for two residually co-Whittaker $A[G]$-modules $V$ and $V'$ in this class, we say that $V$ dominates $V'$ if  there exists a surjection of $V$ onto $V'$ which induces an isomorphism between $\Hom_{W(k)[G]}(W,V)$ and $\Hom_{W(k)[G]}(W,V')$. This gives rise to an equivalence relation, where $V$ is equivalent to $V'$ if there exists a residually co-Whittaker $A[G]$-module $V''$ such that $V''$ dominates both $V$ and $V'$. Then $W$ is universal for residually co-Whittaker $A[G]$-modules in the sense that for any residually co-Whittaker module $V$, the action of $E$ on $\Hom_{W(k)[G]}(W,V)$ yields a map from $E$ to $A$ and $W\otimes_EA$ dominates $V$.\\
We now prove that the class of residually co-Whittaker modules satisfies the hypothesis of the main results of this paper.
 \begin{proposition}
 Let $(V,\pi)$ be a residually co-Whittaker $A[G]$-module of type $(K_x,U_x,\psi)$. Then $(V,\pi)$ satisfies the assumptions of Theorem \ref{theorat} and \ref{functionaltheorem}.
 \end{proposition}
\begin{proof}
By definition $(V,\pi)$ is admissible, so to show that the assumptions of Theorem \ref{theorat} hold in this case, it remains to prove that $(V,\pi)$ is $G$-finite. However, this follows since the element $w_0\in W$ defined by
$$w_0(g)=\begin{cases*}
\psi(g)&$g\in K'$,\\
0 &\text{else,}
\end{cases*}$$
generates $W$ as an $W(k)[G]$-module and the map $W\otimes_{W(k)}\Hom_{W(k)[G]}(W,V)\to V$ is surjective.\\
We proceed by showing that $(V,\pi)$ satisfies the assumptions of Theorem \ref{functionaltheorem}. First we will prove that the canonical trace map $\Phi\colon V\otimes_A\widetilde{V}\to A$ is surjective. Note that by Frobenius reciprocity we obtain
$$\Hom_{W(k)[G]}(W,V)\cong\Hom_{W(k)[K']}(\psi,V)\cong V^{K',\psi},$$
where $V^{K',\psi}=\{v\in V\mid\pi(k)v=\psi(k)v\text{ for all }k\in K'\}$. In particular since by assumption $\Hom_{W(k)[G]}(W,V)$ is a free rank one $A$-module the same holds for $V^{K',\psi}$. Consider now the projection morphism $e_{K',\psi}(v)=\mu(K')^{-1}\int_{K'}\psi(k)^{-1}kvdk$ which maps $V$ onto $V^{K',\psi}$. Since $V^{K',\psi}$ is a free $A$-module we can choose a surjective morphism in $\Hom_A(V^{K',\psi},A)$ and precomposing with $e_{K',\psi}$ yields a surjective element $\lambda$ of $\Hom_A(V,A)$. Now any compact open subgroup on which $\psi$ is trivial also stabilizes $\lambda$ and we obtain that $\lambda\in\widetilde{V}$.\\
It remains to show that the $A$-module $$\Hom_{A[G\times G]}(V\otimes_A\widetilde{V},A)\cong\Hom_{A[G]}(V,\widetilde{\widetilde{V}})$$ is free of rank one. We fix a basis element $v_0$ of $V^{K',\psi}$. Moreover, let $W_A\coloneqq\cInd_{K'}^G(\psi)\otimes A=\cInd_{K'}^G(\psi\otimes A)$ and by abuse of notation we denote the element $w_0\otimes 1$ by $w_0$. By assumption the map $W_A\to V$ defined by $g\cdot w_0\otimes 1\mapsto \pi(g)v_0$ is surjective. 
%Note that $\Hom_{W(k)[G]}(W,V)\cong \Hom_{A[G]}(W_A,V)$. 
As we proved above $\Phi$ is surjective and hence $A\Phi$ is a free rank one submodule of $\Hom_{A[G\times G]}(V\otimes_A\widetilde{V},A)$. The surjection $W_A\to V$ mentioned above yields an embedding
\begin{equation}\label{eq:embpsi}\Hom_{A[G\times G]}(V\otimes_A\widetilde{V},A)\hookrightarrow\Hom_{A[G\times G]}(W_A\otimes_A\widetilde{V},A).\end{equation}
We can compute the latter 
\begin{align*}
\Hom_{A[G\times G]}(W_A\otimes_A\widetilde{V},A)&\cong\Hom_{A[G]}\left(W_A,\widetilde{\widetilde{V}}\right)\\
&\cong\Hom_{A[K']}\left(\psi\otimes A,\widetilde{\widetilde{V}}\right)\\
&\cong\left(\widetilde{\widetilde{V}}\right)^{K',\psi}.
\end{align*}
For any smooth $A[G]$-module $U$ and $\theta\in\widetilde{U}$ we have that $(e_{K',\psi}\theta)(u)=\theta(e_{K',\psi^{-1}}u)$ and hence $\left(\widetilde{U}\right)^{K',\psi}=\Hom_A(U^{K',\psi^{-1}},A)$. We obtain
$$\left(\widetilde{\widetilde{V}}\right)^{K',\psi}\cong\Hom_A(\widetilde{V}^{K',\psi^{-1}},A)\cong\Hom_A(\Hom_A(V^{K',\psi},A),A)$$
which is isomorphic to $A$ since $V^{K',\psi}\cong A$. Hence we have the following chain of morphisms
$$A\hookrightarrow\Hom_{A[G\times G]}(V\otimes_A\widetilde{V},A)\hookrightarrow\Hom_{A[G\times G]}(W_A\otimes_A\widetilde{V},A)\cong A.$$
We will show that $1$ gets sent to $1$ under the above chain of morphisms which implies that $\Hom_{A[G\times G]}(V\otimes_A\widetilde{V},A)$ is free of rank one.\\
Note that $\widetilde{V}^{K',\psi^{-1}}$ is isomorphic to $A$ and a basis element is given by $\tilde{v_0}$ which is characterized by the property that $\tilde{v_0}(v_0)=1$ and that it vanishes on $(1-e_{K',\psi})V$. Analogously we have a basis element $\tilde{\tilde{v_0}}$ of $\left(\widetilde{\widetilde{V}}\right)^{K',\psi}$ which satisfies $\tilde{\tilde{v_0}}(\tilde{v_0})=1$ and vanishes on $(1-e_{K',\psi^{-1}})\widetilde{V}$. Under the isomorphism $$\left(\widetilde{\widetilde{V}}\right)^{K',\psi}\cong\Hom_{A[G\times G]}(W_A\otimes_A\widetilde{V},A)$$ we have that $\tilde{\tilde{v_0}}$ maps to the element that is defined by $g w_0\otimes\lambda\mapsto\tilde{\tilde{v_0}}(g^{-1}\lambda)$.
Under ($\ref{eq:embpsi}$) the trace map $\Phi$ gets sent to the map in $\Hom_{A[G\times G]}(W_A\otimes_A\widetilde{V},A)$ which is defined by $g\cdot w_0\otimes\lambda\mapsto\lambda(\pi(g)v_0)$. For all $g\in G$ and $\lambda\in\widetilde{V}$ we can write $\widetilde{\pi}(g^{-1})\lambda=a\tilde{v_0}+r$ for some $a\in A$ and $r\in(1-e_{K',\psi^{-1}})\widetilde{V}$. Hence we obtain that
$$\lambda(\pi(g)v_0)=(\widetilde{\pi}(g^{-1})\lambda)(v_0)=a\tilde{v_0}(v_0)=a=\tilde{\tilde{v_0}}(a\tilde{v_0})=\tilde{\tilde{v_0}}(\widetilde{\pi}(g^{-1})\lambda),$$
which shows that $\Phi$ and $\tilde{\tilde{v_0}}$ get mapped to the same element in $\Hom_{A[G\times G]}(W_A\otimes_A\widetilde{V},A)$. This finishes the proof.
\end{proof}
\section{Normalizing Factor}

The construction of the gamma factor is unfortunately not complete. We need to include a certain normalizing factor for the intertwining operator (in particular such that the intertwining operator does not depend on a choice of a measure on the unipotent radical $N$ of $P$). Since $W(k)$ contains enough $p$-power roots of unity we can choose a nontrivial smooth additive character $\psi\colon F\to W(k)^\times$. We follow the article \cite{karel1979functional} by Karel.\\
We exclude the odd orthogonal case (see Section 6 of \cite{lapid2005local} for the necessary changes). Let $\mathfrak g^\Box$ be the Lie algebra of $G^\Box$ which we identify with
$$\{X\in\End(W^{\Box})\mid h^\Box(w,w'X)+h^\Box(wX,w')=0\text{ for all } w,w' \in W^\Box\}.$$
Moreover, we denote the Lie algebra of $\overline{N}$, the unipotent radical of the parabolic $\overline{P}$ opposite to $P$, by $\overline{\mathfrak n}$. Under the above identification $\overline{\mathfrak n}$ corresponds to 
$$\{X\in\mathfrak g^\Box\mid\ker(X)\supseteq W^\bigtriangledown\supseteq \operatorname{Im}(X)\}.$$ 
This induces an isomorphism $\overline{N}\to\overline{\mathfrak n}$ given by $u\mapsto u-I$. If we are not in the odd orthogonal case we can choose an element $B$ of the Lie algebra $\mathfrak n$ of the unipotent radical $N$ of $P$ which has maximal rank. For any such $B$ we define a character $\psi_B$ of $\overline{\mathfrak n}$ by
$$\psi_B(X)=\psi(\operatorname{Trd}_{W^\Box}(XB))$$
for $X\in\overline{\mathfrak n}$ where the product $XB$ is taken in $\End_D(W^\Box)$. Via the isomorphism mentioned above this also defines a character of $\overline{N}$. The map $w\mapsto (w,w)$ allows us to identify $W$ with $W^\bigtriangleup$ and similarly $W^\bigtriangledown$ with $W$ via $(w,-w)\mapsto w$. Since any $B\in\mathfrak n$ induces a linear map from $W^\bigtriangledown$ to $W^\bigtriangleup$, via this identifications we can make sense of $\operatorname{Nrd}_W(B)$.

We have a filtration on $I(X,\omega)=i_{P}^{G^\Box}(\omega_X\circ\Delta)$ given by the Bruhat decomposition $G^\Box=\coprod_{w\in^PW^{\overline{P}}}Pw\overline{P}$, which we already described in Section \ref{sectionintertwining}. The bottom element consists of those functions which have support in the open orbit $P\overline{P}$ which we denote by $I_1$. For $w\in ^PW^{\overline{P}},w\not=1$ we denote by $I_w$ the quotient of those elements supported in $\bigcup_{w'\leq w}Pw\overline{P}$ modulo those functions which are supported in $\bigcup_{w'<w}Pw\overline{P}$. Then $I_w$ is isomorphic to the smooth vectors in the $\overline{P}$-representation
$$\left\{f\colon Pw\overline{P}\to A[X,X^{-1}]\mid f(sx)=\omega_X(\Delta(s))f(x)\text{ for all }s\in P,x\in Pw\overline{P}\text{ and }P\backslash\operatorname{supp}(f)\text{ is compact}\right\}.$$

%\left\{
%f\colon Pw\overline{P}\to A[X,X^{-1}] \biggm| \begin{array}{l}
%f(sx)=\omega_X(\Delta(s))f(x)\text{ for all }s\in P,x\in Pw\overline{P}\\
%P\backslash\operatorname{supp}(f)\text{ is compact.}
%\end{array}
%\right\}
\begin{proposition}
There is a natural isomorphism 
\begin{equation}\label{keyiso}
(I_1)_{\overline{N},\psi_B^{-1}}\cong I(X,\omega)_{\overline{N},\psi_B^{-1}},
\end{equation}
which is induced by the inclusion $I_1\hookrightarrow I(X,\omega)$.
\end{proposition}
\begin{proof}
 Since $\overline{N}$ is a pro-$p$ group the functor $(\_)_{\overline{N},\psi_{B}^{-1}}$ is exact and hence to show (\ref{keyiso}) it is enough to prove that $(I_w)_{\overline{N},\psi_{B}^{-1}}=0$ for $w\not=1$. By I.4.11 of \cite{vigneras1996representations} this is equivalent to showing that for any $\xi$ in $I_w,w\not=1$ we can find an open compact subgroup $U$ of $\overline{N}$ such that \begin{equation}\label{zerocoin}
\int_U(u\cdot\xi)(g)\psi_B(u)du=0
\end{equation}
for all $g\in Pw\overline{P}$. Since for $u'\in U$ we have that $\int_U(u\cdot\xi)(gu')\psi_B(u)du=\psi_B(u')\int_U(u\cdot\xi)(g)\psi_B(u)du$ it is enough to show Equation (\ref{zerocoin}) for $g\in\operatorname{supp}(\xi)$.\\ 
The condition that $B$ has maximal rank implies that $\psi_B$ is nontrivial on $\overline{N}\cap x^{-1}Px$ for all $x\in G^\Box$ which are not contained in $P\overline{P}$. By Lemmas 2.5 and 2.6 of \cite{karel1979functional} we can find an open compact subgroup $U\subseteq\overline{N}$ such that for each $\xi\in I_w$ and $x\in\operatorname{supp}(\xi)$ we can find an $r_x\in U\cap x^{-1}Px$ such that $\psi_B(r_x)\not=1$.
We compute
$$\int_U\xi(xu)\psi_B(u)du=\int_U\xi(xr_xu)\psi_B(r_xu)du=\psi_B(r_x)\omega_X(\Delta(xr_xx^{-1}))\int_U\xi(xu)\psi_B(u)du$$
for $x\in\operatorname{supp}(\xi)$.
However since $xr_xx^{-1}$ is a unipotent element its projection to $M$ is contained in the derived subgroup of $M$ and thus $\omega_X(\Delta(xr_xx^{-1}))=1$. Since $\psi_B(r_x)$ is a $p$-th power root of unity we see that $1-\psi_B(r_x)$ is invertible in $W(k)$ which shows Equation (\ref{zerocoin}).
\end{proof}

For $f\in I_1$ we define 
$$l_{\psi_B}(f)\coloneqq\int_{\overline{N}}f(u)\psi_B(u)du,$$
and this integral is well-defined since $f\mid_{\overline{N}}$ has compact support (in particular $l_{\psi_B}(f)\in A[X,X^{-1}]$). Note that $l_{\psi_B}$ factors through $(I_1)_{\overline{N},\psi_B^{-1}}$. Hence by composing the canonical projection $I(X,\omega)\to I(X,\omega)_{\overline{N},\psi_{B}^{-1}}$ and the isomorphism of Equation (\ref{keyiso}) with $l_{\psi_B}$ we obtain a map from $I(X,\omega)$ to $A[X,X^{-1}]$ which we also denote by $l_{\psi_B}$.\\
We will now see that the just constructed map satisfies a functional equation.

\begin{proposition}
	There is an element $c(X,\omega,B,\psi)\in S^{-1}(A[X,X^{-1}])$ such that for all $f\in I_1$ we have
	$$l_{\psi_B}(M_X(f))=c(X,\omega,B,\psi)l_{\psi_B}(f).$$
\end{proposition}
\begin{proof}
	Clearly, $l_{\psi_B}$ is a an element of
	\begin{equation}\label{homkarel}
	\Hom_{\overline{N}}(I_1,\psi_B^{-1}\otimes A[X,X^{-1}]).
	\end{equation}
	Note that $I_1$ is isomorphic to the smooth compactly supported $A[X,X^{-1}]$-valued functions on $\overline{N}$, which is the same as $\cInd_{1}^{\overline{N}}(A[X,X^{-1}])$. Since the contragredient of $\psi_B^{-1}\otimes A[X,X^{-1}]$ is $\psi_B\otimes A[X,X^{-1}]$, by Chapter I., Section 5.11 of \cite{vigneras1996representations} and Frobenius reciprocity we obtain that
	\begin{align*}
	\Hom_{\overline{N}}\left(I_1,\psi_B^{-1}\otimes A[X,X^{-1}]\right)&\cong\Hom_{\overline{N}}\left(\psi_B\otimes A[X,X^{-1}],\widetilde{\cInd_{1}^{\overline{N}}(A[X,X^{-1}])}\right)\\
	&\cong\Hom_{\overline{N}}\left(\psi_B\otimes A[X,X^{-1}],\Ind_{1}^{\overline{N}}(A[X,X^{-1}])\right)\\
	&\cong\Hom_{A[X,X^{-1}]}\left(A[X,X^{-1}],A[X,X^{-1}]\right)\\
	&\cong A[X,X^{-1}].
	\end{align*}
	Now $\overline{M}_X\mid_{I_1}\colon I_1\to I(X^{-1},\overline{\omega}^{-1})$ composed with $l_{\psi_B}$ yields another element of (\ref{homkarel}) and since we can find (similarly to Proposition \ref{exone}) an element $f_0$ of $I_1$ such that $l_{\psi_B}(f_0)=1$ there is a unique element $\hat{c}\in A[X,X^{-1}]$ such that 
	$$l_{\psi_B}(\overline{M}_X(f))=\hat{c}l_{\psi_B}(f)$$
	for all $f\in I_1$. We define $c(X,\omega,B,\psi)\coloneqq\widehat{R}^{-1}\hat{c}\in S^{-1}(A[X,X^{-1}])$ and obtain the functional equation
	$$l_{\psi_B}(M_X(f))=c(X,\omega,B,\psi)l_{\psi_B}(f)$$
	for all $f\in I_1$.
\end{proof}
By the following proposition we see that $c(X,\omega,B,\psi)$ behaves nicely under base change.
\begin{lemma}\label{normalizingbase}
Let $A'$ be a noetherian $W(k)$-algebra and suppose we have an $W(k)$-algebra homomorphism $\alpha\colon A\to A'$. Let $\hat{\alpha}\colon A[[X]][X^{-1}]\to A'[[X]][X^{-1}]$ be the canonical map induced by $\alpha$. Then for all $f\in I(X,\omega)$ we have
$\hat{\alpha}(l_{\psi_B}(f))=l_{\psi_B}(\hat{\alpha}(f))$ and in particular
$$\hat{\alpha}(c(X,\omega,B,\psi)=c(X,\omega\otimes A',B,\psi).$$
\end{lemma}
To obtain the normalized gamma factor we need to divide by $c(X,\omega,B,\psi)$ and hence we need to show that it is a unit in $S^{-1}(A[X,X^{-1}])$. 
\begin{proposition}\label{normisunit}
	The normalizing factor $c(X,\omega,B,\psi)$ is a unit in $S^{-1}(A[X,X^{-1}])$.
\end{proposition}
To show the above result we will prove an explicit formula for $c(X,\omega,B,\psi)$. Note that, for example by the approach taken in \cite{bushnell2006local}, it is easy to see that Tate gamma factors can be defined for any $A$-valued smooth character. In more detail, for characters $\omega\colon E^\times\to A^\times$ and $\psi\colon E\to W(k)^\times$ there is a $\gamma(X,\omega,\psi)\in S^{-1}A[X,X^{-1}]$ such that if the residue field $\kappa(\mathfrak p)$ at a prime ideal $\mathfrak p\in\operatorname{Spec}(A)$ is contained in $\mathbb C$ we have that $\gamma(X,\omega,\psi)\mod\mathfrak p\in\mathbb C(X)$ equals the gamma factor defined by Tate. Hence all terms in the following Proposition \ref{normcomp} are well-defined.\\
 Let $w_1,\dotsc,w_n$ be a basis of $W$. Then we set $R=(h(w_i,w_j))_{i,j}\in\operatorname{GL}_n(D)$ and let 
$$\Theta(W,h)=(-1)^n\operatorname{Nrd}(R)\in E^\times/E^{\times2}$$
be the discriminant of the pair $(W,h)$. For any element $\delta\in F^\times/F^{\times2}$ let $\chi_\delta\colon F^\times\to \{\pm1\}$ be given by $x\mapsto(x,\delta)_F$ where $(.,.)_F$ is the Hilbert symbol of $F$. We set $\delta(A)\coloneqq (-1)^n\operatorname{Nrd}_{W}(A)$. Moreover, in the case (I3) we denote by $\eta$ the unique nontrivial character of $F^\times$ with kernel $N_{E/F}(E^\times)$ and set $\vartheta(W,h)=\eta\left((-1)^{n(n-1)/2}\det((h(w_i,w_j))_{i,j})\right)$. We let $e(G)$ be the invariant of Kottwitz, where $e(G)=1$ if $D$ is split and if $D$ is not split
$$e(G)=\begin{cases*}
(-1)^{n(n+1)/2}& \text{(I1) \& (I2) and} $\epsilon=1$,\\
(-1)^{n(n-1)/2}& \text{(I1) \& (I2) and} $\epsilon=-1$,\\
(-1)^n &\text{ in the linear case.}
\end{cases*}$$
Let
\begin{multline*}R(X,\omega,B,\psi)=\\\begin{cases*}
\omega_X(\operatorname{Nrd}_W(B))^{-1}\gamma(Xq_F^{-1/2},\omega\chi_{\delta(A)},\psi)\epsilon(q_F^{-1/2},\chi_{\delta(A)},\psi)^{-1}&\text{(I1) \& (I2) and }$\epsilon=1$,\\
\omega_X(\operatorname{Nrd}_W(B))^{-1}\epsilon(q_F^{-1/2},\chi_{\Theta(W,h)},\psi)&\text{(I1) \& (I2) and }$\epsilon=-1$,\\
\omega_X(\operatorname{Nrd}_W(B))^{-1}\eta(\det(R))&\text{(I3),}\\
\omega_X(\operatorname{Nrd_W(B/2)})^{-2}&\text{ in the linear case.}
\end{cases*}\end{multline*}
\begin{proposition}\label{normcomp} Let
	\begin{enumerate}
		\item for the cases (I1) \& (I2) where $\epsilon=1$
		$$d(X,\omega,B,\psi)=\frac{e(G)X^{2n\operatorname{val}_F(2)}\abs{2}_F^{n(n-1/2)}\omega(4)^{-n}}{\omega(\operatorname{Nrd}(R))\abs{\operatorname{Nrd}(R)}_F^{n+1/2}\gamma(Xq_F^{n-1/2},\omega,\psi)\prod_{i=0}^{n-1}\gamma(X^2q_F^{2i},\omega^2,\psi)}R(X,\omega,B,\psi),$$
		\item for the cases (I1) \& (I2) where $\epsilon=-1$
		$$d(X,\omega,B,\psi)=\frac{e(G)X^{2n\operatorname{val}_F(2)}\abs{2}_F^{n(n-1/2)}\omega(4)^{-n}}{\omega(\operatorname{Nrd}(R))\abs{\operatorname{Nrd}(R)}_F^{n-1/2}\prod_{i=0}^{n-1}\gamma(X^2q_F^{2i},\omega^2,\psi)}\omega_X(\operatorname{Nrd}_W(B))^{-1},$$
		\item in the case (I3)
		$$d(X,\omega,B,\psi)=\frac{\gamma_{\text{Weil}}(\psi^{-1}\circ N_{E/F})^{n(n-1)/2}}{\prod_{r=0}^{n-1}\gamma\left(Yq_F^{n-1-r},(\omega|_{F^\times})\cdot\eta^r,\psi\right)}\omega_X(\operatorname{Nrd}_W(B))^{-1},$$ where $Y$ equals $X^2$ if $E/F$ is unramified and $Y=X$ otherwise. Moreover $\gamma_{\text{Weil}}(\psi^{-1}\circ N_{E/F})$ is the Weil index of the character of second degree $\psi^{-1}\circ N_{E/F}$,
		\item in the linear case (II)
		$$d(X,\omega,B,\psi)=\frac{(-1)^nX^{4n\operatorname{val}_F(2)}\omega(4)^{-2n}}{\prod_{i=0}^{2n-1}\gamma(X^2q_F^i,\omega^2,\psi)}R(X,\omega,B,\psi).$$
	\end{enumerate}
Then we can choose an appropriate Haar measure on $\overline{N}$ such that 
$$c(X,\omega,B,\psi)=d(X,\omega,B,\psi).$$
\end{proposition}

For any smooth character $\omega\colon E^\times\to A^\times$ there is some $e\in\mathbb N$ such that $\omega$ restricted to 
$$\mathcal O_E^{\times,e}=\{u\in\mathcal O_E^\times\mid u\equiv 1\mod\varpi_E^e\}$$
is trivial. We have a universal character
\begin{align*}
\omega_e^{\text{univ}}\colon E^\times/\mathcal O_E^{\times,e}&\to W(k)\left[E^\times/\mathcal O_E^{\times,e}\right]\\
x&\mapsto x,
\end{align*}
together with a ring homomorphism $f_\omega\colon W(k)\left[E^\times/\mathcal O_E^{\times,e}\right]\to A$ such that $\omega=f_\omega\circ \omega_e^{\text{univ}}$.
\begin{proof}[Proof of Proposition \ref{normcomp}]

First note that if $A$ are the complex numbers then $c(X,\omega,B,\psi)=d(X,\omega,B,\psi)$ holds, see Section 3 of \cite{kudla1997degenerate} for the case (I3), and Proposition 4.2 of \cite{kakuhama2019local} for the other cases. See also these references for the definition of the used Haar measure on $\overline{N}$.\\
We proceed by extending the result to $\omega_e^{\text{univ}}$. Note that $W(k)\left[E^\times/\mathcal O_E^{\times,e}\right]$ is $\ell$-torsion free and reduced. Moreover, it is uncountable and hence for any minimal prime $\mathfrak p$ of $W(k)\left[E^\times/\mathcal O_E^{\times,e}\right]$ the residue field at $\mathfrak p$, which we denote by $\kappa(\mathfrak p)$, has characteristic zero and is at most uncountable. This implies that the algebraic closure $\overline{\kappa(\mathfrak p)}$ can be embedded into $\mathbb C$ and we identify $\overline{\kappa(\mathfrak p)}$ with a subfield of $\mathbb C$. For any $\mathfrak p\in\operatorname{Spec}\left(W(k)\left[E^\times/\mathcal O_E^{\times,e}\right]\right)$ let $\alpha_{\mathfrak p}\colon W(k)\left[E^\times/\mathcal O_E^{\times,e}\right]\to\kappa(\mathfrak p)$ be the canonical map. Hence we obtain
	$$c(X,\omega^{\text{univ}}_e\otimes\kappa(\mathfrak p),B,\psi)=d(X,\omega^{\text{univ}}_e\otimes\kappa(\mathfrak p),B,\psi).$$
	By Lemma \ref{normalizingbase} we have
	$$\hat{\alpha}_{\mathfrak p}(c(X,\omega^{\text{univ}}_e,B,\psi))=c(X,\omega^{\text{univ}}_e\otimes\kappa(\mathfrak p),B,\psi)$$
	and since gamma factors behave well under base change that
	$$\hat{\alpha}_{\mathfrak p}(d(X,\omega^{\text{univ}}_e,B,\psi))=d(X,\omega^{\text{univ}}_e\otimes\kappa(\mathfrak p),B,\psi).$$
	Hence 
	$$c(X,\omega^{\text{univ}}_e,B,\psi)-d(X,\omega^{\text{univ}}_e,B,\psi)$$
	lies in the kernel of $\hat{\alpha}_{\mathfrak p}$ and in particular all coefficients of this Laurent series are elements of $\mathfrak p$ for all minimal prime ideals $\mathfrak p$ of $W(k)\left[E^\times/\mathcal O_E^{\times,e}\right]$. Since this ring is reduced the intersection of all minimal prime ideals is trivial which shows the proposition for $\omega_e^\text{univ}$. The general case follows via base change by the map $f_\omega$.
\end{proof}
\begin{proof}[Proof of Proposition \ref{normisunit}]
Any Tate gamma factor satisfies
$$\gamma\left(\frac{1}{q_FX},\omega^{-1},\psi\right)\gamma(X,\omega,\psi)=\omega(-1),$$ 
which implies that it is invertible in $S^{-1}A[X,X^{-1}]$. All the other factors in the expression of $d(X,\omega,B,\psi)$ are clearly invertible in $S^{-1}A[X,X^{-1}]$ which implies the result.
\end{proof}
\subsection{The normalized gamma factor}
We assume that $(\pi,V)$ admits a central character $z_\pi$.
\begin{definition}
For any smooth $A[G]$-module $(\pi,V)$ satisfying the conditions in Theorems \ref{theorat} and \ref{functionaltheorem} we define the normalized gamma factor associated to $\pi\times\omega$ as 
$$\gamma(Xq_F^{-1/2},\pi\times\omega,\psi)=z_{\pi}(-1)\Gamma(Y,\pi,\omega)c(Y,\omega,B,\psi)^{-1}R(Y,\omega,B,\psi)\in S^{-1}A[X,X^{-1}]$$
where $Y=X^2$ in the case (I3) for $E/F$ unramified and $Y=X$ in all other cases.
\end{definition}
\begin{remark}Note that by Proposition \ref{normcomp} this does not depend on the choice of $B$.
\end{remark}

\bibliography{references}
\bibliographystyle{amsalpha}
\end{document}